\documentclass[12pt]{amsart}
 \usepackage{amsmath,amssymb,enumerate,amsfonts,amsthm,graphicx,color}
 \usepackage[all]{xy}
 \usepackage[colorlinks=true, linkcolor=red, linktoc=page, citecolor=blue]{hyperref}
  \usepackage{tikz}
\newtheorem{thm}{Theorem}[section]
\newtheorem{cor}[thm]{Corollary}
\newtheorem{lem}[thm]{Lemma}

\theoremstyle{definition}
\newtheorem{defn}[thm]{Definition}
\theoremstyle{definition}
\newtheorem{rem}[thm]{Remark}
\theoremstyle{definition}

\theoremstyle{definition}
 \newtheorem{note}[thm]{Notation}
 \newtheorem{conj}[thm]{Conjecture}
 
  \newcommand{\OO}{\mathcal{O}}
  \newcommand{\PP}{\mathbb{P}}
  \newcommand{\II}{\mathcal{I}}
  \newcommand{\NN}{\mathbb{N}}

  \newcommand{\X}{\widetilde{X}}
   \newcommand{\Xp}{\widetilde{X'}}
  \newcommand{\Xz}{\widetilde{X''}}
   
  \newcommand{\T}{\widetilde{T}}

 \newcommand{\XX}{\mathcal{X}}
 \newcommand{\DD}{\mathfrak{D}}
 \newcommand{\CC}{\mathfrak{C}}
  
  \newcommand{\Ch}{\widehat{\mathfrak{C}}}
 \newcommand{\la}{\lambda}

\newcommand{\HH}{\mathbf{H}}

\begin{document}

\title{On the Hartshorne--Hirschowitz theorem}
\author{T.  Aladpoosh}
\address{ Tahereh Aladpoosh; School of Mathematics, Institute for Research in Fundamental
Sciences (IPM),  P.O.Box: 19395-5746, Tehran, Iran.}
\email{tahere.alad@ipm.ir}

\author{M. V.  Catalisano}
\address{Maria Virginia Catalisano;  Dipartimento di Ingegneria Meccanica, Energetica,
Gestionale e dei Trasporti, Universit\`{a} degli Studi  di Genova,
 Genoa, Italy.} \email{catalisano@dime.unige.it}


\begin{abstract}
The Hartshorne--Hirschowitz  theorem  says that a generic union of
lines in $\PP^n$, $(n\geq 3)$,  has good postulation. The proof of
Hartshorne and Hirschowitz in the initial case  $\PP^3$ is difficult
and so long, which is handled by a method of specialization via a
smooth quadric surface with the property of having two rulings of
skew lines.
 We provide a proof in the case $\PP^3$  based on a new degeneration
 of disjoint
lines via a plane $H\cong\PP^2$, which we call \emph{$(2,s)$-cone
configuration}, that is a schematic union of $s$ intersecting lines
passing through a single point $P$ together with the trace of an
$s$-multiple point supported at $P$ on the double plane $2H$.
 In the first part of this paper, we discuss our degeneration
 inductive approach. We prove that a $(2,s)$-cone configuration
 is a degeneration of $s$ disjoint lines
 in $\PP^3$, or more generally in $\PP^n$.
 In the second part of the paper, we use this degeneration in an
 effective method  to show that a generic union of lines in $\PP^3$
 imposes independent conditions on the linear system $|\OO_{\PP^3}(d)|$ of surfaces of
  given degree $d$.
 The basic motivation
behind  our degeneration approach   is that
  it  looks more systematic that gives some hope of  extensions  to the analogous
  problem in higher dimensional  spaces, that is the postulation problem for $m$-dimensional planes in
  $\PP^{2m+1}$.
\end{abstract}
\maketitle
 \tableofcontents
\section{Introduction}
Given a closed subscheme $X\subset\PP^n$,  we say that $X$  has
{\emph{good postulation}} or {\emph{maximal rank}}
 if
   $X$ imposes the expected number of conditions to
  hypersurfaces of any degree. This is equivalent to saying that for
  each $d\geq 0$
  one or other of the integers  $h^0(\II_X(d)),~ h^1(\II_X(d))$
 is zero.
This problem is equivalent to computing the Hilbert function of $X$.
 Let $HF(X,d)$ be the Hilbert function of $X$ in
degree $d$, namely,  $HF(X,d)= h^0(\OO_{\PP^n}(d))- h^0(\II_X(d))$.
There is an expected value for the Hilbert function of $X$ in degree
$d$ given by a naive count of conditions. This value is determined
by assuming that $X$ imposes independent conditions on the linear
system $|\OO_{\PP^n}(d)|$, i.e.,
$$h^0(\II_X(d))= \max
\left\{h^0(\OO_{\PP^n}(d))-h^0(\OO_X(d)),0\right\},$$
 which implies
that $X$ has good postulation in degree $d$.

When we restrict our attention to the special class of schemes
$X\subset \PP^n$ which supported  on  unions of generic linear
spaces
 there is much interest in
the postulation problem  (see e.g. \cite{GMR}, \cite{HH},
\cite{CCG1}, \cite{Bal} for  reduced case, and \cite{CCG4},
\cite{AB}, \cite{B1}, \cite{Alad}, \cite{bau} for  non-reduced
case), yet surprisingly very little is known about them, even in the
reduced case.
 Specifically concerning the class of reduced schemes of
 generic linear spaces,
 the first obvious case is to take $X$ a generic collection of points
in $\PP^n$, according to \cite{GMR} it is well known that $X$  has
good postulation.
 In the next case concerning a  generic collection of lines in $\PP^n$,  there is
 a spectacular theorem  by
  R. Hartshorne and A. Hirschowitz \cite{HH}, going back to 1981, which states that:
 \begin{thm}[Hartshorne--Hirschowitz]\label{HH th}
 Let  $X\subset\PP^n$, $(n\geq 3)$,  be a generic union of $e$ lines.
Then $X$ has good postulation, i.e.,
$$h^0(\II_X(d)) = \max \left \{ {d+n \choose n}- e(d+1), 0 \right \}.$$
\end{thm}
Inspired by these results about points and skew lines, E. Carlini,
M. V. Catalisano and A. V. Geramita proposed a conjecture on the
postulation of generic disjoint unions of linear spaces \cite[\S
1]{CCG1}, which says that:
\begin{conj}[Carlini--Catalisano--Geramita]\label{conj}
If $X\subset\PP^n$ is  a generic union of linear spaces with
non-intersecting components, then $X$ has good postulation.
\end{conj}
As we have mentioned above, this conjecture is true for $\dim X= 0$,
where only points are involved, and for $\dim X= 1$, i.e. a generic
collection of lines and points, where we have
Hartshorne--Hirschowitz theorem about generic lines and also we know
how adding generic points to a scheme can still preserve its good
postulation \cite{GMR}. As soon as we go up to $\dim X> 1$, the
problem becomes more and more complicated. In fact, when $\dim X>1$
the results in the literature
 are  so little  and  the conjecture remains widely
open even for $\dim X=2$ (see e.g. \cite{CCG1} and \cite{Bal} for a
generic union of lines and a few planes).

 The basic natural step in the proof of Conjecture \ref{conj} for
$\dim X>1$,  would be to
 provide  an analogue of  Hartshorne--Hirschowitz
 theorem for a generic
 collection  of  planes in $\PP^n$,   $(n\geq 5)$,
  that seems to be extremely
 difficult, and even surprisingly enough, one may hope to generalize this
 approach to the  case of $m$-dimensional planes,  ($m$-planes for short), in $\PP^n$, $(n\geq
 2m+1)$.
 Actually, the
 postulation problem for $m$-planes is so difficult
 that one seldom expects to solve it completely, yet which provides
 stimulus for a great amount of efforts, and which enable us to make
 progress on this problem.

Nevertheless,  if we wish to go further in this direction we need to
analyze  the proof of Hartshorne and Hirschowitz.
  Interestingly enough,  their proof in the initial case $\PP^3$ is difficult and long by
using degeneration techniques via    a smooth quadric surface, which
occupies more than
 half of the length of the paper \cite{HH}.
 Indeed, one
aspect of the results in their paper   is the postulation of
degenerated schemes, which plays an important role in the proof of
their main theorem (Theorem \ref{HH th}). Roughly speaking, what we
mean  by \emph{degeneration} of a scheme, is a limiting scheme
inside a projective space of a flat family of original ones.
 Apart from the fact that,
as $X$ varies in a flat family, by the semicontinuity theorem  for
cohomology groups \cite[III, 12.8]{Hart}, the condition of good
postulation is an open condition on the family of $X$,
 one may use  degenerations and the
 semicontinuity theorem to investigate that $X$ has  good
 postulation. This  means that,
to prove  that $X$ has good postulation it is enough to find a
degeneration of $X$ which has good postulation.
 It is the degeneration that requires some artistry and a lot of
 technical details, which usually gives the most problems.
 It is precisely in
 this part of the procedure that we will  give some new
 ideas.
 In fact, the main goal of this paper is:  firstly  introducing  a new
degeneration of $s$ disjoint lines in $\PP^3$ (or even in $\PP^n$),
which we will call \emph{$(2,s)$-cone configuration}, and secondly
verifying the postulation of generic unions of lines by applying
successfully  this  method of degeneration. An  important feature of
the $(2,s)$-cone configuration  is that
 it will  be extremely useful
when we will
 attack the postulation problem of  generic
lines in $\PP^3$ using  degenerations via   a  $\PP^2$ instead of
the
 smooth quadric surface.

 The paper is organized as follows.
 Section \ref{bb} contains  preliminary material.
 In Section \ref{s2}, we first introduce the notion of
$(2,s)$-cone configuration (see \S \ref{def}); next we show that a
$(2,s)$-cone configuration is a flat limit  of a family of disjoint
unions of $s$ lines, that is a degeneration of $s$ skew lines (see
\S \ref{ext}, specially  Figure \ref{fig1},  for the case $s=3$  and
\S \ref{de2s} for the general case).
 In Section \ref{pr},
with the purpose of  stating  a proof of Hartshorne--Hirschowitz
theorem in the case of $\PP^3$  using this new method of
degeneration,
 we reformulate Theorem \ref{HH th} to a good postulation
statement $\HH_d$ (see \S \ref{st}); then  we propose  two other
good postulation  statements $\HH'_{d-1}$  and $\HH''_{d}$, which
are necessary for our inductive approach (see \S \ref{hhh}). Finally
 Sections \ref{h'}, \ref{h''} give the proofs  of the statements
  $\HH'_{d-1}$ and $\HH''_{d}$, respectively.

 We would like to finish Introduction by mentioning that,
with an eye towards
  handling the postulation problems of  disjoint unions of
   $m$-planes  in $\PP^n$ $(n\geq 2m+1)$,
   one motivation for us comes
from the fact  about the
 lack of  hypersurfaces in
the initial case $\PP^{2m+1}$  with geometric constructions analogue
to the smooth quadric
 surface in $\PP^3$; while hyperplanes have  constructions
   analogue to the plane.   This   leads us to provide a proof of
 Hartshorne--Hirschowitz theorem in $\PP^3$ using degenerations by a
 plane instead of the quadric surface.
 Actually,
 in analogy with the degeneration of $s$ disjoint lines,
we believe that the notion  of $(2,s)$-cone configuration can be
extended, somehow, for $s$ disjoint $m$-planes, which appears to be
ambitious and requires the most sophisticated investigations,
 and then  one may hope to generalize our approach to the postulation problem of
  $m$-planes.

\section{Preliminaries and Notations}\label{bb}
In this paper we  work  over an algebraically closed field $k$ with
characteristic zero.

 Given a closed subscheme $X$
of $\PP^n$, $I_X$ and $\mathcal{I}_X$ will denote the homogeneous
ideal and the  ideal sheaf of $X$, respectively.

 If  $X, Y$ are closed subschemes of $\PP^n$ and $X\subset Y$,
 then we denote by $\mathcal{I}_{X,Y}$ the
 ideal sheaf of $X$ in $\mathcal{O}_Y$.

 If  $X$ and $Y$  are two closed subschemes of $\PP^n$, we denote by
$X+Y$ the  schematic union of $X$ and $Y$, i.e. the subscheme of
$\PP^n$ defined by the ideal sheaf $\II_{X}\cap \II_{Y}\subset
\OO_{\PP^n}$.

 If $\mathcal{F}$ is a coherent sheaf on the scheme $X$, for any
integer $i\geq 0$ we use $h^i(X,\mathcal{F})$ to denote the
$k$-vector space dimension of the cohomology group
$H^i(X,\mathcal{F})$.
 In particular,
when $X= \PP^n$, we will often omit $X$ and we will simply write
$h^i(\mathcal{F}).$

\vspace{0.2cm}
The basic tool for the  study of the postulation problem is the so
called \emph{Castelnuovo's inequality} (for proof we refer to
\cite[Section 2]{AH} or \cite{AH2}).

We first recall the notion of residual scheme \cite[\S 9.2.8]{Ful}.
\begin{defn}
Let $X, Y$ be closed subschemes of $\PP^n$.
\begin{itemize}
\item [(i)]
The closed subscheme of $\PP^n$ defined by the ideal sheaf
$(\II_X:\II_{Y})$ is called the {\textbf{residual} } of $X$ with
respect to $Y$ and denoted by
 $Res_Y(X)$.
\item [(ii)]
The schematic intersection $X\cap Y$ defined by the ideal sheaf
 $(\II_X+\II_{Y})/\II_{Y}$ of $\OO_{Y}$ is called
 the {\textbf{trace} } of $X$ on $Y$ and denoted by  $Tr_Y(X)$.
\end{itemize}
\end{defn}
We note that the generally valid identity for ideal sheaves
$$(\II_{X_1}\cap\II_{X_2}:\II_{Y})=
(\II_{X_1}:\II_{Y})\cap(\II_{X_2}:\II_{Y})$$
 implies that the residual of the schematic union $X_1+X_2$ is the
 schematic union of the residuals.
\begin{lem}[Castelnuovo's Inequality]\label{cas}
 Let $d,e\in \NN$, and $d\geq e$. Let $H\subseteq{\PP^n}$
be a hypersurface of degree $e$, and let $X\subseteq{\PP^n}$ be a
closed subscheme. Then
$$h^0(\PP^n,\II_X(d))\leq
h^0(\PP^n,\II_{Res_H(X)}(d-e))+h^0(H,\II_{Tr_{H}(X)}(d)).$$
\end{lem}

\vspace{0.2cm}
 The following remark is quite immediate.
\begin{rem}\label{ss'}
Let $X= X_1+\cdots+X_s\subset \PP^n$ be the union of
non-intersecting closed subschemes $X_i$. Let $s'<s$ and
$$ X'= X_1+\cdots+X_{s'}\subset X.$$
\begin{itemize}
\item [(i)]
If $h^1(\II_X(d))= 0$, then $h^1(\II_{X'}(d))=0.$
\item [(ii)]
If  $h^0(\II_{X'}(d))= 0$, then $h^0(\II_{X}(d))=0.$
\end{itemize}
\end{rem}

\vspace{0.2cm}
 A \emph{(fat) point of multiplicity $m$}, or an
\emph{$m$-multiple point}, with support $P\in \PP^n$, denoted $mP$,
is the zero-dimensional subscheme of $\PP^n$ defined  by the ideal
sheaf $(\II_P)^m$, i.e. the $(m-1)^{th}$ infinitesimal neighborhood
of $P$. In particular, if $m= 2$ we shall call the scheme $2P$ a
\emph{ double point} with support $P$, i.e. the first infinitesimal
neighborhood of $P$.

In case $P\in X$ for any smooth variety $X\subset\PP^n$, we will
write $mP|_{X}$ for the $(m-1)^{th}$ infinitesimal neighborhood of
$P$ in $X$, that is the schematic intersection of the $m$-multiple
point $mP$ of $\PP^n$ and $X$ with $(\II_{P,X})^m$ as its ideal
sheaf.

More generally, if $\mathbb{X}= \{P_1,\ldots,P_s\}$ is any set of
distinct points in $\PP^n$  and $m_1,\ldots,m_s$ are positive
integers, then we will denote by $m_1P_1+\cdots+m_sP_s$ the
zero-dimensional subscheme of $\PP^n$  defined by the ideal sheaf
$(\II_{P_1})^{m_1}\cap\ldots\cap (\II_{P_s})^{m_s}$, and we
sometimes refer to it as a \emph{fat point scheme} with support
$\mathbb{X}$.

\vspace{0.2cm}
 We recall a spectacular  result due to Alexander and
Hirschowitz, on the postulation of generic collections of double
points.
\begin{thm}[Alexander-Hirschowitz]\cite{AH}\label{AH th}
Let $n,d\in \NN$, and let $X\subset\PP^n$ be a generic collection of
$s$ double points. Then $X$ has good postulation in degree $d$
except in the following cases:
\begin{itemize}
\item $d=2, ~2\leq s\leq n$;
\item $n=2, ~d=4,~ s=5$;
\item $n=3, ~d=4,~ s=9$;
\item $n=4, ~d=3,~ s=7$;
\item $n=4, ~d=4,~ s=14$.
\end{itemize}
\end{thm}
Since we will apply Alexander-Hirschowitz theorem in the case of
$\PP^2$ frequently in Section \ref{h''}, it is convenient to restate
it as follows.
\begin{cor}\label{ah2}
The scheme $X\subset \PP^2$ consisting of $s$ generic double points
always has good postulation in degree $d$, except for the two cases
$\{s=2, d=2\}$ and $\{s=5, d= 4\}$.
\end{cor}

\vspace{0.1cm}
 If $X$ is a zero-dimensional scheme, we denote by
$\ell(X)$ its length. So in the simple case $X= mP\subset\PP^n$ we
have $\ell(X)= {n+m-1 \choose n}$.

 Even though in this paper we will only use 2-dots, it seemed nice to  give the more general
 definition as follows:
a $\delta$-\emph{dot} is a subscheme of length $\delta$ of a double
point. (Hence a double point in $\PP^n$ is an $(n+1)$-dot.)

From \cite{curv} we know that a generic union of $2$-dots in $\PP^2$
has good postulation. Now the next lemma gives more information on
the postulation problem regarding 2-dots in $\PP^2$. Precisely the
lemma affirms that a generic union  of $2$-dots together with a
multiple point in $\PP^2$ has good postulation, which geometrically
means that a generic collection of 2-dots  imposes the expected
number of  conditions on the linear system of curves passing through
a multiple point.
\begin{lem}\label{dot}
Let $s,m,d \in \NN$, and $d\geq m-1$. Let  the scheme
$X\subset\PP^2$ be a generic union of an $m$-multiple point $mP$ and
$s$ $2$-dots $Z_1,\ldots,Z_s$. Then $X$ has good postulation, i.e.
$$h^0(\II_X(d)) = \max \left \{ {d+2 \choose 2}- {m+1 \choose 2}- 2s, 0 \right \},$$
\end{lem}
\begin{proof}

 Let
 $$t= \left \lfloor {{d+2 \choose 2} -  {m+1 \choose 2}}\over {2}\right
 \rfloor; \ \ \ \  u= \left \lceil {{d+2 \choose 2} -  {m+1 \choose 2}}\over {2}\right
 \rceil.$$
 In order to apply  Remark \ref{ss'},  we must  prove the lemma  only for the two cases $s=t$ and $s=u$.

 If ${d+2 \choose 2} -  {m+1 \choose 2}$ is even,   we have $t=u$.
  If ${d+2 \choose 2} -  {m+1 \choose
 2}$ is odd,   we observe that for  $s= t$ the expected value for
 $h^0(\II_X(d))$ is $1$.  So, if we prove that $h^0(\II_X(d))= 1$
 for $s=t$, we then  conclude  that $h^0(\II_X(d))= 0$ for $s= u$, i.e. the
 expected one.
 Therefore it is enough to prove the lemma only for $s=t$.

If $m= 0$, the scheme $X$ is  generic union of $2$-dots, and the
conclusion immediately follows from
 \cite[Proposition 4.2]{curv}.

 Now we assume that $m>0$.  We will proceed by induction on $m$.

 Let $L$ be a line passing through the point $P$.
We consider two cases.

\vspace{0.3cm}
 \textbf{Case 1.}  When $d-m$ is odd. Let  $t'=
\frac{d-m+1}{2}$, (note that $t\geq t'$). Consider the scheme $\X$
obtained from $X$ by specializing
 the $2$-dots $Z_1,\ldots,Z_{t'}$
so that $Z_i\subset L$, $(1\leq i\leq t')$, (the $2$-dots
  $Z_{t'+1},\ldots,Z_t$ remain  generic not lying on $L$).

 Since $\deg(\X\cap L)= m+2t'= d+1$, the line $L$ is a fixed
  component for the curves of degree $d$ containing $\X$, so we get
  \begin{equation}\label{od}
  h^0(\II_{\X}(d))= h^0(\II_{Res_L(\X)}(d-1)),
  \end{equation}
 where  $Res_L(\X)$ is the  union of the $(m-1)$-multiple
 point $(m-1)P$ and  $2$-dots $Z_{t'+1},\ldots,Z_{t}$.
 Thus by the induction hypothesis we have
  \begin{eqnarray*}
 h^0(\II_{Res_L(\X)}(d-1))&=& \max \left \{ {d+1 \choose 2}- {m \choose 2}- 2(t-t'), 0 \right
 \}\\
 &=& \max \left \{ {d+2\choose 2}- {m+1\choose 2}- 2t, 0 \right\},
\end{eqnarray*}
which by the equality (\ref{od}) gives
$$h^0(\II_{\X}(d)) = \max \left \{ {d+2 \choose 2}- {m+1 \choose 2}- 2t, 0 \right \},$$
 and from here, by
semicontinuity of cohomology \cite[III, 12.8]{Hart},  we see that
$$h^0(\II_X(d)) = \max \left \{ {d+2 \choose 2}- {m+1 \choose 2}- 2t, 0 \right \},$$
hence this case is done.

\vspace{0.2cm}
 \textbf{Case 2.} When $d-m$ is even. Let $t'=
\frac{d-m}{2}$, (note that $t\geq t'+1$).
 We specialize the
$2$-dots  $Z_1,\ldots,Z_{t'}$ so that $Z_i\subset L$,  $(1\leq i\leq
t')$;  moreover, we specialize
 $Z_{t'+1}$ in such a way that the support of $Z_{t'+1}$ is contained in the line $L$, but
 $Z_{t'+1}\not\subset L$, which implies that
  $\mathrm{deg}(Z_{t'+1}\cap L)= 1$.
 Let $\X$ be the specialized scheme, (note that the
     $2$-dots
  $Z_{t'+2},\ldots,Z_t$ remain  generic not lying on $L$).

  Since $\deg(\X\cap L)= m+2t'+1= d+1$, the line $L$ is a fixed
  component for the curves of degree $d$ containing $\X$, so  we get
  \begin{equation}\label{ev}
  h^0(\II_{\X}(d))= h^0(\II_{Res_L(\X)}(d-1)),
  \end{equation}
 where, by  noting that $Res_L(Z_{t'+1})$ is a simple point,
 the scheme $Res_L(\X)$ is the generic union of the $(m-1)$-multiple
 point $(m-1)P$,  $2$-dots $Z_{t'+2},\ldots,Z_{t}$ and one
 simple point. Hence  by the induction hypothesis we have
 \begin{eqnarray*}
 h^0(\II_{Res_L(\X)}(d-1))&=& \max \left \{ {d+1 \choose 2}- {m \choose 2}- 2(t-t'-1)-1, 0 \right
 \}\\
 &=& \max \left \{ {d+2\choose 2}- {m+1\choose 2}- 2t, 0 \right\},
\end{eqnarray*}
then by the equality (\ref{ev}) and  the semicontinuity of
cohomology \cite[III, 12.8]{Hart} we get the conclusion
$$h^0(\II_X(d)) = \max \left \{ {d+2 \choose 2}- {m+1 \choose 2}- 2t, 0 \right \},$$
which finishes the proof.
\end{proof}
\section{Main Argument via Degeneration; \\ $(2,s)$-Cone
Configuration}\label{s2} A natural approach to the postulation
problem is to argue by degeneration.
 In view of  the fact that we have the semicontinuity
theorem for cohomology
 groups in a flat family \cite[III, 12.8]{Hart}, one may use the degenerations and the
 semicontinuity theorem in order to be able to better handle
 the
postulation of schemes supported on  generic unions of linear
spaces. Specifically, if one can prove that the property of having
good postulation is satisfied in the special fiber, i.e. the
degenerate scheme, then one may hope to  obtain the same property in
the general fiber, i.e. the original scheme.
\subsection{Definitions and Basic  Constructions}\label{def}
In the celebrated paper \cite{HH} Hartshorne and Hirschowitz
investigated a new degeneration technique to attack the postulation
problem for a generic union of lines.  In fact,  they degenerate two
skew lines in $\PP^3$ in such a way that the resulting scheme
becomes  a ``degenerate conic with an embedded point" (which also
was used in \cite{Hir}).
 Even more generally, one can push this trick of ``adding nilpotents" further,
  to give a degeneration of two skew lines in higher dimensional
  projective spaces $\PP^n,$ $n\geq 3$,   this is what
 the authors introduced in \cite[Definition 2.7 with $m=1$]{CCG1} and called a
  {\em (3-dimensional) sundial}.

 According to the terminology of \cite{HH}, we say that $C$ is a
{\it degenerate conic} if  $C$ is the union of two intersecting
lines $L$ and $M$, so $C=L+M$.

Now we recall the definition of a 3-dimensional sundial or simply a
sundial (see \cite[Definition 3.7]{CCG3} or \cite[Definition 2.7
with $m=1$]{CCG1}).
\begin {defn}\label{sundial}
Let $L$ and $M$ be two intersecting lines in $\mathbb P^n, ~n \geq
3$, and let
 $T\cong \Bbb P^{3}$ be a generic linear space
containing the degenerate conic $L+M$. Let
 $P$ be the singular point of  $L+M$,
i.e. $P = L \cap M$.   We call the scheme $L+M+ 2P|_T$ a {\it
degenerate conic with an embedded point} or a {\it (3-dimensional)
sundial}.
\end {defn}
 One can show a sundial is a flat limit inside $\PP^n$ of a flat
 family whose general fiber is the disjoint union of two lines, i.e.
 a sundial is a degeneration of two generic lines in $\PP^n$, $n\geq 3$. This is the
 content of the following lemma (see \cite[Example 2.1.1]{HH} for the case $n=3$, and
  \cite[Lemma 3.8]{CCG3} or \cite[Lemma 2.5 with $m=1$]{CCG1}  for the general case $n\geq 3$).
\begin{lem}\label{sun}
Let $\XX_1\subset\PP^n$, $n\geq 3$, be the disjoint union of two
lines $L_1$ and $M$. Then there exists a flat family of subschemes
$\XX_{\lambda}\subset \langle \XX_1\rangle\cong \PP^3$, $(\lambda\in
k)$, whose general fiber is
 the union of two skew lines and whose special fiber is the sundial $\XX_0= M+L+2P|_{\langle \XX_1\rangle}$,
 where $L$ is a line and $M\cap L= P$.
\end{lem}
\begin{defn}\label{DD}
A \emph{$(2,m;n)$-point}  is a zero-dimensional scheme in $\PP^n$,
$n\geq 3$,  with support at one point $P$, and whose ideal sheaf is
of type $\II_P^m+\II_H^2$, where $H\subset\PP^n$ is a plane
containing $P$. We denote it by $\DD_{H,m}(P)$,   or $\DD_m(P)$ if
no confusion arises.
 So $\DD_{H,m}(P)$ is  the trace of $mP$ on the double plane $2H$,
$$\DD_{H,m}(P)= mP\cap 2H.$$
\end{defn}
If we take the residual and the trace of the $(2,m;n)$-point
$\DD_{H,m}(P)$ with respect to $H$ we observe that the residual
scheme $Res_H(\DD_{H,m}(P))$ is  the subscheme $(m-1)P|_H$ of
$\PP^n$, that is $(m-2)^{th}$ infinitesimal neighborhood of $P$ in
$H$,  and the trace scheme $Tr_H(\DD_{H,m}(P))$ is  the $m$-multiple
point $mP$ of $H$.
\begin{defn}\label{type}
We say that $\mathcal{C}$   is a \emph{cone configuration of type
$s$}, $s\geq 3$,  if $\mathcal{C}$ is the union of $s$ intersecting
lines $L_1,\ldots,L_s$ passing through a single point $P$, so
$\mathcal{C}= L_1+\cdots+L_s$ and $L_1\cap\ldots\cap L_s= P$.
\end{defn}
With the help of Definitions \ref{DD} and \ref{type},  we make the
following definition, which is essential to our degeneration
approach and is central to the proof of our main theorem  in Section
\ref{pr}.
\begin{defn}\label{2scone}
Let $L_1,\ldots,L_s$ be $s$  lines in $\PP^n$ lying on a plane and
 intersecting in a single point $P$, i.e.
 $L_1\cap\ldots\cap L_s= P$, and $n,s\geq 3$.
  Let $H\cong \PP^2$ be the plane
containing the cone configuration $L_1+\cdots+L_s$.
 We call the scheme
 $L_1+\cdots+L_s+\DD_{H,s}(P)$ a \emph{cone configuration of type $s$ with a $(2,s;n)$-point}
 or  simply a  \emph{$(2,s)$-cone configuration}.
\end{defn}
Then for the $(2,s)$-cone configuration
$L_1+\cdots+L_s+\DD_{H,s}(P)$ the residual with respect to $H$ will
be   $(s-1)P|_H$, while the trace with respect to $H$ will be the
cone configuration $L_1+\cdots+L_s$.

 Our interest in  cone configurations, particularly $(2,s)$-cone
 configurations,
 arises from the important fact that
 these constructions  are  occurring as resulting schemes of
 specific degenerations of disjoint
 lines; however   not in a straightforward   way but rather in
    a   less obvious and subtle way.
   This  feature of $(2,s)$-cone configurations is  interesting and surprising  in its own and
plays a key
   role in this paper.
\begin{note}
Let $L_1,\ldots,L_s$ be $s\geq 2$ lines in $\PP^2$ such that no
three lines meet in a point.  The set  of ${s \choose 2}$ points
which are the pairwise intersections of the lines $L_i$
 is called a \emph{star
configuration} of points. We will also call the scheme
$L_1+\cdots+L_s$   a \emph{star configuration} of lines.
\end{note}
\subsection{An Example: The Case of (2,3)-Cone
Configuration}\label{ext}
 If we consider in $\PP^3$ a flat family of
subschemes whose general scheme  is the disjoint union of three
lines, while    the three lines meet at  one single  point  in a
special fiber, then  the special scheme
 consists of three intersecting lines together with a $(2,3)$-point,
 i.e. a $(2,3)$-cone configuration. That is a remarkable result which
we verify in this example.

 Let $L_1,L_2,L_3$ be three disjoint lines
in $\PP^3$, and let $H$ be a fixed  plane containing the line $L_1$.
 First We degenerate  the  lines $L_2,L_3$ in such a way that they together with $L_1$  become
$3$ lines in  a star configuration in $H$. We will again denote
these specialized lines by $L_2,L_3$, moreover,  denote by $P, Q, R$
the three intersection points of the $3$ lines $L_1,L_2,L_3$, namely
$$P= L_1\cap L_2;\ \ \ \  Q= L_1\cap L_3;\ \ \ \  R= L_2\cap L_3.$$
 So according to Lemma \ref{sun} we get the star configuration $L_1+L_2+L_3$ in $H$
together with three double points, i.e. the specialized scheme is
$$L_1+L_2+L_3+2P+2Q+2R\subset\PP^3,$$
which we call $\XX_1$ (see Figure \ref{fig1}).

Now for the calculation. Choose coordinates $t,x,y,z$ on $\PP^3$. We
may assume that $P=[1:0:0:0];\ Q=[1:1:0:0]; \ R=[1:0:1:0]$, and  the
plane $H$ is defined by the equation $z=0$ (we can always choose
coordinates so that this is the case). This implies that
\begin{displaymath}
  \left\{
\begin{array}{lll}
I_P=(x,y,z); \\
I_Q=(x-t,y,z); \\
I_R=(y-t,x,z).
\end{array}
\right.
 \hspace{1cm}
 \left\{
\begin{array}{lll}
I_{L_1}=(y,z); \\
I_{L_2}=(x,z); \\
I_{L_3}=(x+y-t,z).
\end{array}
\right.
\end{displaymath}

 Up to now, we found that the scheme $\XX_1= L_1+L_2+L_3+2P+2Q+2R$,
 which is a star configuration of $3$ lines having the double
 structure at their intersection points,
 is a degeneration of three disjoint lines in $\PP^3$.
Next  we want to make a further degeneration, precisely the
degeneration of $\XX_1$ in which the points $Q=[1:1:0:0]$ and
$R=[1:0:1:0]$ move to the point $P=[1:0:0:0]$. Now what is the
resulting scheme of this degeneration? To answer this, let the line
$L_3$ move in  such a way that the two points $Q$ and $R$ approach
the point $P$.
 We will show that this movement occurs in a flat family
 $\{\XX_{\la}\}$  and the
  limiting scheme $\XX_0$  is a cone configuration
of type $3$ union the $(2,3)$-point $3P\cap 2H$ (Figure \ref{fig1}).
 Then  we can
realize  the limiting scheme as the scheme formed by $L_1+L_2+2P$
and the limit of $L_3+2Q+2R$ along with the aforementioned
direction.

We will now calculate the flat family  $\{\XX_{\la}\}_{\la\in k}$
just described, whereas    $\XX_{\la}$ is the scheme
$$L_1+L_2+L_{3,\la}+2P+2Q_{\la}+2R_{\la},$$
 where  the  two points $Q_{\la}=[1:\la:0:0]$ and  $R_{\la}=[1:0:\la:0]$
  move to the point $P=[1:0:0:0]$, as well as the line joining $L_{3,\la}$,
when $\la$ tends to $0$, that is
\begin{displaymath}
\left\{
\begin{array}{lll}
I_{Q_{\la}}=(x-\la t,y,z); \\
I_{R_{\la}}=(y-\la t,x,z); \\
I_{L_{3,\la}}=(x+y-\la t,z).
\end{array}
\right.
\end{displaymath}

Then the homogeneous ideal of $\XX_{\la}$ is
\begin{eqnarray*}
I_{\XX_{\la}}&=&(y,z)\cap(x,z)\cap(x+y-\la t,z)\cap \\
& &(x,y,z)^2\cap(x-\la t,y,z)^2\cap(y-\la t,x,z)^2\\
&=& (x^2y+xy^2-\la txy,z)\cap \\
& & (x,y,z)^2\cap(x-\la t,y,z)^2\cap(y-\la t,x,z)^2.
 \end{eqnarray*}
 A straightforward computation yields
 $$(x,y,z)^2\cap(x-\la t,y,z)^2\cap(y-\la t,x,z)^2 =$$
 $$(x^2y^2,xyz,x^2z-\la txz,y^2z-\la tyz,x^2y+xy^2-\la txy,x^4-\la tx^3+\la^2t^2x^2,y^4-\la
 ty^3+\la^2t^2y^2,z^2),$$
 and from here we get
\begin{eqnarray*}
I_{\XX_{\la}}&=&(x^2y+xy^2-\la txy) +\\
& & \left[(z)\cap (x^2y^2,xyz,x^2z-\la txz,y^2z-\la tyz,x^4-\la
tx^3+\la^2t^2x^2,y^4-\la
 ty^3+\la^2t^2y^2,z^2)\right]\\
 &=&(x^2y+xy^2-\la txy) +\\
& & (xyz,x^2z-\la txz,y^2z-\la tyz,x^4z-\la
tx^3z+\la^2t^2x^2z,y^4z-\la ty^3z+\la^2t^2y^2z,z^2).
 \end{eqnarray*}
 From this, setting $\la=0$, we obtain the ideal of $\XX_0$, which
 is
 $$I_{\XX_0}= (x^2y+xy^2,xyz,x^2z,y^2z,z^2).$$
On the other hand the scheme $L_1+L_2+L_{3,0}+\DD_{H,3}(P)$,
recalling that $\DD_{H,3}(P)= 3P\cap 2H$, is defined by the ideal
\begin{eqnarray*}
 I_{L_1}\cap I_{L_2}\cap
 I_{L_{3,0}}\cap[I^3_P+I^2_H]&=&(y,z)\cap(x,z)\cap(x+y,z)\cap[(x,y,z)^3+(z^2)]\\
 &=&
(x^2y+xy^2,z)\cap(x^3,y^3,x^2y,xy^2,x^2z,y^2z,xyz,z^2)\\
&=&(x^2y+xy^2,x^2z,y^2z,xyz,z^2),
\end{eqnarray*}
which is exactly $I_{\XX_0}$. Thus we see that the limiting scheme
$\XX_0$ is the $(2,3)$- cone configuration
$L_1+L_2+L_{3,0}+\DD_{H,3}(P)$, as we desired (see Figure
\ref{fig1}).

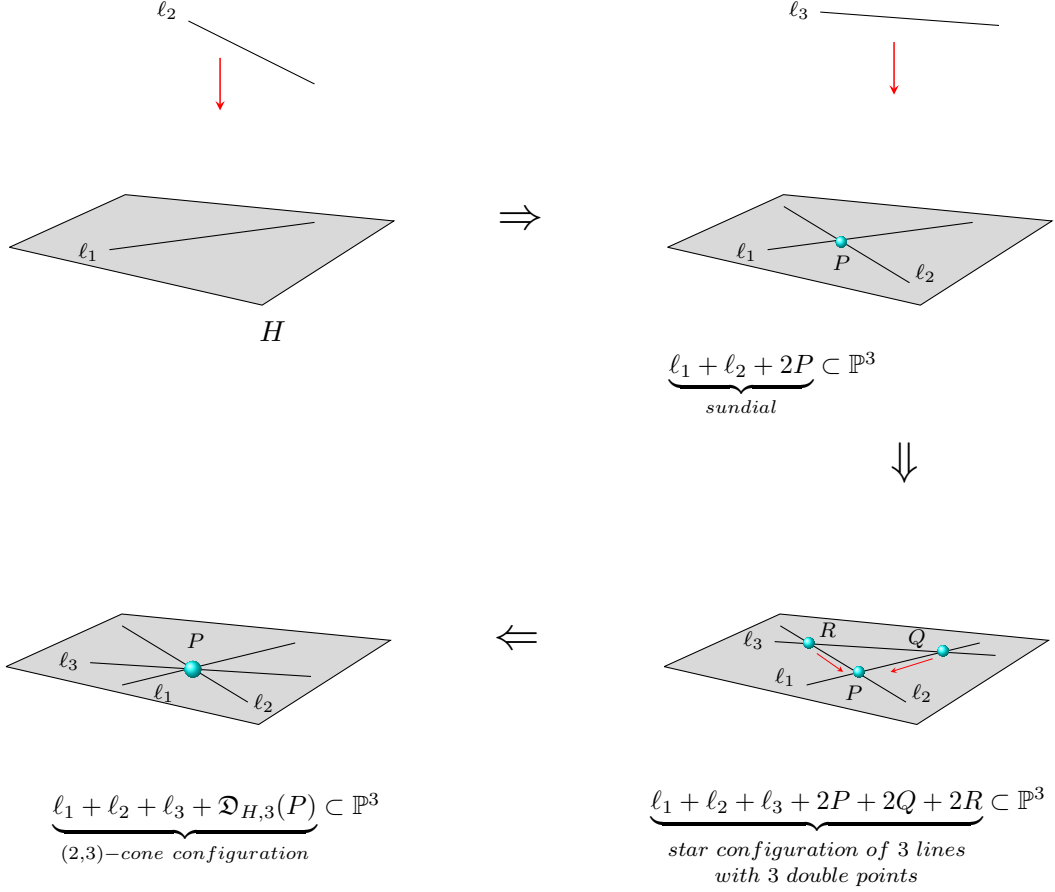
\begin{figure}[h]
\begin{tikzpicture}[scale=0.7,>=stealth]
 \filldraw[fill=gray!30] (0,1.1) -- (4.8,0) -- (7.3,1.6) -- (2.2,2.1) -- (0,1.1);
  \node (A1) at (1.5,1) {$\scriptstyle \ell_1$};
 \node (A2) at (6,1.6)  {$\scriptstyle $} ;
  \draw (A1) -- (A2);
 \node (B1) at (3,5.6) {$\scriptstyle \ell_2$};
 \node (B2) at (6,4.1) {$\scriptstyle $};
 \draw (B1) -- (B2);
 \draw(5,-0.5) node{\small{$H$}};
 \draw [red] [->, line width=0.02cm] (4,4.7) -- (4,3.7);
 \draw [black] (9.7,1.6) node{\Large{$\Rightarrow$}};
  \filldraw[fill=gray!30] (12.5,1.1) -- (17.3,0) -- (19.8,1.6) -- (14.7,2.1) -- (12.5,1.1);
  \node (A'1) at (14,1) {$\scriptstyle \ell_1$};
 \node (A'2) at (18.5,1.6)  {$\scriptstyle $} ;
  \draw (A'1) -- (A'2);
 \node (B'1) at (14.5,2) {$\scriptstyle $};
 \node (B'2) at (17.3,0.3) {$\scriptstyle $};
  \draw (17.4,0.6) node{\tiny{$\ell_2$}};
 \draw (B'1) -- (B'2);
  \draw (15.8,0.8) node{\tiny{$P$}};
 \shade[ball color=cyan] (15.8,1.2) circle (3pt);
   \node (C1) at (15,5.6) {$\scriptstyle\ell_3 $};
 \node (C2) at (19,5.3) {$\scriptstyle $};
  \draw (C1) -- (C2);
  \draw [red] [->, line width=0.02cm] (16.8,5) -- (16.8,4);
  \draw (14.5,-1.5) node{\Small{$\underbrace{\ell_1+\ell_2+2P}_{sundial}\subset\PP^3$}};
  \draw [black] (17,-3) node{\Large{$\Downarrow$}};
   \end{tikzpicture}

   \vspace{1.5cm}
\begin{tikzpicture}[scale=0.7,>=stealth]
 \filldraw[fill=gray!30] (12.5,1.1) -- (17.3,0) -- (19.8,1.6) -- (14.7,2.1) -- (12.5,1.1);
 \node (A'1) at (15,0.7) {$\scriptstyle $};
 \node (A'2) at (18.7,1.6)  {$\scriptstyle $} ;
 \draw (A'1) -- (A'2);
 \node (B'1) at (14.5,2) {$\scriptstyle $};
 \node (B'2) at (17.3,0.3) {$\scriptstyle $};
 \draw (17.4,0.6) node{\tiny{$\ell_2$}};
 \draw (14.8,0.9) node{\tiny{$\ell_1$}};
 \draw (B'1) -- (B'2);
 \node (C'1) at (14.2,1.6) {$\scriptstyle \ell_3$};
 \node (C'2) at (19,1.3) {$\scriptstyle $};
 \draw (C'1) -- (C'2);
  \draw (16.1,0.6) node{\tiny{$P$}};
 \shade[ball color=cyan] (16.2,1) circle (3pt);
 \draw (15.6,1.8) node {\tiny{$R$}};
 \shade[ball color=cyan] (15.25,1.55) circle (3pt);
  \draw (17.3,1.62) node {\tiny{$Q$}};
 \shade[ball color=cyan] (17.8,1.4) circle (3pt);
  \draw [red] [->, line width=0.001cm] (17.6,1.25) -- (16.8,1);
  \draw [red] [->, line width=0.001cm] (15.4,1.35) -- (15.9,1);
 \draw (16,-2.2) node{\Small{$\underbrace{\ell_1+\ell_2+\ell_3+2P+2Q+2R}_{\tiny{\begin{array}{c}
                                                                            star~ configuration~ of~ 3~ lines \\
                                                                            with~ 3~ double~ points
                                                                          \end{array}}}
   \subset\PP^3$}};
\draw [black] (9.7,1.6) node{\Large{$\Leftarrow$}};
 \filldraw[fill=gray!30] (0,1.1) -- (4.8,0) -- (7.3,1.6) -- (2.2,2.1) -- (0,1.1);
 \node (A''1) at (2,0.7) {$\scriptstyle $};
 \node (A''2) at (5.7,1.6)  {$\scriptstyle $} ;
 \draw (A''1) -- (A''2);
 \node (B''1) at (2,2) {$\scriptstyle $};
 \node (B''2) at (4.8,0.3) {$\scriptstyle $};
 \draw (B''1) -- (B''2);
  \node (C''1) at (1.2,1.2) {$\scriptstyle \ell_3$};
 \node (C''2) at (6,0.9) {$\scriptstyle $};
 \draw (C''1) -- (C''2);
  \draw (3,0.6) node{\tiny{$\ell_1$}};
 \draw (4.9,0.4) node{\tiny{$\ell_2$}};
  \draw (3.6,1.6) node{\tiny{$P$}};
 \shade[ball color=cyan] (3.55,1.05) circle (4.5pt);
  \draw (4,-2) node{\Small{$\underbrace{\ell_1+\ell_2+\ell_3+\mathfrak{D}_{H,3}(P)}_{(2,3)-cone~ configuration}\subset\PP^3$}};
 \end{tikzpicture}
 \caption{The degeneration process to get a $(2,3)$-cone configuration}
 \label{fig1}
 \end{figure}

 Summing up, this example illustrates that  one can
degenerate $3$ disjoint lines in such a way that the resulting
scheme becomes a $(2,3)$-cone configuration.  Actually, we first
degenerate $3$ disjoint lines to
 a scheme of type $\XX_1=
L_1+L_2+L_3+2P+2Q+2R$, next we construct a flat family
$\{\XX_{\la}\}_{\la\in k}$ in such a way that  whose general fiber
$\XX_{\la}= L_1+L_2+L_{3,\la}+2P+2Q_{\la}+2R_{\la}$ approaches  the
$(2,3)$-cone configuration  $\XX_0= L_1+L_2+L_{3,0}+\DD_{H,3}(P)$
when $\la\rightarrow 0$.
 This means  that a $(2,3)$-cone
configuration can be viewed as a degeneration of a star
configuration of $3$ lines with $3$ double points, consequently, as
a degeneration of $3$ disjoint lines.

In addition, we easily find  that
\begin{eqnarray*}
I_{\XX_0}:(z)&=& (x^2y+xy^2,xyz,x^2z,y^2z,z^2):(z)\\
&=& (x^2,y^2,xy,z),
\end{eqnarray*}
that is $[(x,y,z)^2+(z)]$, i.e. the ideal of $2P|_H$, hence
$$Res_H(\XX_0)= 2P|_H.$$
Also,
\begin{eqnarray*}
I_{\XX_0}+(z)&=& (x^2y+xy^2,xyz,x^2z,y^2z,z^2)+(z)\\
&=& (x^2y+xy^2,z),
\end{eqnarray*}
that is $[(y,z)\cap(x,z)\cap(x+y,z)]$, i.e. the ideal of
$L_1+L_2+L_{3,0}$, so
$$Tr_H(\XX_0)= L_1+L_2+L_{3,0}.$$

\vspace{0.2cm}
\subsection{A Degeneration Inductive Method:
 How (2,s)-Cone Configurations
Appear  as  Degenerations of $s$ Skew Lines?}\label{de2s}
 One can generalize the result of  \S \ref{ext}, which dealt with
  the degeneration of $3$ disjoint lines
 exclusively, so as to also deal with the degeneration
     of $s$ disjoint lines, provided  that
 a disjoint union of $s$ lines can be degenerated
  to a $(2,s)$-cone configuration.
 In fact  our method appearing in \S \ref{ext}
   can be successfully  extended to the general case.
 Suppose that $L_1,\ldots,L_s$ be $s$ disjoint lines in $\PP^3$ and
 that $H$ be a fixed plane containing  $L_1$.
 (Note that by abuse of notation if we
 specialize some line  $L_i$, we will again denote it by $L_i$ in the sequel).
 To begin,  we  specialize the line $L_2$ into $H$,
  so that we have the sundial $ L_1+L_2+2P$.
  Next  we proceed in $s-2$ steps by an inductive argument
    based on adding one line at each inductive step,
  specializing it to lie on $H$,   and finally degenerating
   it in an effective way for the purpose of getting desired cone configuration.
 Indeed, starting with the sundial  $ L_1+L_2+2P$,
  we can lift step by step,
  to get the $(2,i)$-cone configuration
  $L_1+\cdots+L_i+\DD_{H,i}(P)$,
  for each $i$, $3\leq i\leq s$.
 Let us  explain the process in details.
 The initial step is exactly  the same as   \S \ref{ext}, which
  consists of specializing the line $L_3$ into $H$
 that  meets  $L_1,L_2$  in 2 distinct points $P_{1,3}= L_1\cap L_3, P_{2,3}= L_2\cap
 L_3$,
 and  moving  $L_3$ so that
 $P_{1,3},P_{2,3}$ approach  the point  $P$.
 Hence  what we get  in  this step is the  $(2,3)$-cone
 configuration $L_1+L_2+L_3+\DD_{H,3}(P)$, as described in \S \ref{ext}.
 In the second step,
 first specialize the line $L_4$ into $H$ so that this line  meets  each of the previous
 lines $L_1,L_2,L_3$  in distinct points $P_{1,4}= L_1\cap L_4$,
 $P_{2,4}= L_2\cap L_4$ and $P_{3,4}= L_3\cap L_4$,
 where,  in view of  Lemma \ref{sun},  the specialized scheme is
 obviously
 $$[L_1+L_2+L_3+\DD_{H,3}(P)]+[L_4+2P_{1,4}+2P_{2,4}+2P_{3,4}].$$
  Then
 let the  line $L_4$ move in such a way that the points
 $P_{1,4},P_{2,4},P_{3,4}$ approach the point $P$,
   so, using the same reasoning as the previous step  (for which we refer to Lemma \ref{cone}),
   we get the $(2,4)$-cone configuration
 $L_1+L_2+L_3+L_4+\DD_{H,4}(P)$.
 Continuing in this way, we arrive at {\large{\textbf{step i}}},  $1\leq i\leq s-2$,  which can be
 performed  as follows:
\begin{itemize}
\item specialize the line $L_{i+2}$ into $H$
so that this line  meets each of the previous lines
 $L_1,\ldots,L_{i+1}$ in  distinct points
 $$P_{1,i+2}=L_1\cap L_{i+2}; \ldots; P_{i+1,i+2}=L_{i+1}\cap L_{i+2};$$
 \item  observe that, in view of Lemma \ref{sun},   the outcome of this specialization is
 $$[L_1+\cdots+L_{i+1}+\DD_{H,i+1}(P)]+[L_{i+2}+2P_{1,i+2}+\cdots+2P_{i+1,i+2}];$$
 \item  move  the  specialized
   line  $L_{i+2}$ in such a way
that the $(i+1)$  points $P_{1,i+2},\ldots,P_{i+1,i+2}$ approach the
point $P$;
\item  find  that
the resulting scheme of this degeneration  is  the
 $(2,i+2)$-cone configuration $$L_1+\cdots+L_{i+2}+\DD_{H,i+2}(P).$$
\end{itemize}
 The appealing fact of the inductive procedure is that
 the way one checks the resulting scheme  in all the  steps is the
same (see Lemma \ref{cone} for a detailed proof).
 Eventually  we  see that in the last step the
 resulting degenerated scheme
  is  the  $(2,s)$-cone configuration
  $$L_1+\cdots+L_s+\DD_{H,s}(P),$$
 that is what we expected.

The above discussion leads  to  the following substantial
improvement of the conclusion  of   \S \ref{ext}, which  asserts
that a $(2,s)$-cone configuration is a flat limit  of a flat
 family whose general fiber is the disjoint union of $s$ lines.
\begin{lem}\label{cone}
Let $\XX\subset\PP^3$  be the disjoint union of $s$ lines
$L_1,\ldots,L_s$, and $s\geq 3$. Then there exists a flat family of
subschemes $\XX_{\lambda}\subset\PP^3$, $(\lambda\in k)$, whose
general fiber is
 the union of $s$ disjoint lines and whose special fiber is the
 $(2,s)$-cone configuration
$\XX_0= L_1+L'_2+\cdots+L'_s+\DD_{H,s}(P)$,
 where  $H\cong\PP^2$ is a generic plane containing the line $L_1$, and
  $L_2',\ldots,L_s'$ are lines in $H$ passing through a point $P$,
 such that  $L_1\cap L_2'\cap\ldots\cap L_s'= P$.
\end{lem}
\begin{proof}
We proceed by induction on $s$. For $s= 3$, the result is immediate
from \S \ref{ext}.

Now suppose that $s\geq 4$. Assume by induction that the lemma holds
for $s-1$, which says that one can degenerate a disjoint union of
$s-1$ lines to a $(2,s-1)$-cone configuration.

We can prove this lemma almost in the same way as in \S \ref{ext}.
However, some
 modifications will be need.
 We will not attempt to do this exhaustively, but will only indicate
 certain aspects of this general situation, and will mention
  what changes one needs to make.

Let $L_1,\ldots,L_s$ be $s$ disjoint lines in $\PP^3$. First, by the
induction hypothesis, one  degenerates the $(s-1)$ lines
$L_1,\ldots,L_{s-1}$ to the $(2,s-1)$-cone configuration
$L_1+L_2'+\cdots+L'_{s-1}+\DD_{H,s-1}(P)$, where  $H$ is  a fixed
plane containing the line $L_1$ and $P$ is the single intersection
point of the lines $L_1,L_2',\ldots,L_{s-1}'$.
 Considering the
coordinates $t,x,y,z$ on $\PP^3$, one  may suppose that
$P=[1:0:0:0]$, and $H$ is the plane defined by the equation $z=0$,
 moreover  the lines $L_1,L_2',\ldots,L_{s-1}'$ are
defined by the ideals
\begin{equation}\label{li}
\left\{
\begin{array}{llllll}
I_{L_1}=(x,z); \\
I_{L_2'}=(y,z); \\
I_{L_3'}=(x+y,z); \\
I_{L_4'}=(x+2y,z);\\
\vdots \\
I_{L_{s-1}'}=(x+(s-3)y,z),
\end{array}
\right.
\end{equation}
(note that we can always degenerate the lines  so that this is the
case).

Next,  specialize the line $L_s$ into $H$ such that this line  meets
each of the previous lines $L_1,L_2',\ldots,L_{s-1}'$ in the
following distinct points
$$P_1= L_1\cap L_s; \  P_2= L_2'\cap L_s; \  P_3= L_3'\cap L_s;  \ldots;  P_{s-1}= L_{s-1}'\cap
L_s.$$
 Notice  that we have the double structure at these intersection
 points, (see Lemma \ref{sun}), so this implies that the specialized scheme becomes
  $$[L_1+L_2'+\cdots+L_{s-1}'+\DD_{H,s-1}(P)]+[L_s+2P_1+\cdots+2P_{s-1}],$$
 which we call  $\XX_1$.

We may assume   that the line $L_s$ is defined by
 $$I_{L_s}= (x-y-t,z).$$
hence we have
\begin{displaymath}
\left\{
\begin{array}{llllll}
P_1= [1:0:-1:0]; \\
P_2= [1:1:0:0]; \\
P_3= [1:1/2:-1/2:0]; \\
P_4= [1:2/3:-1/3:0];\\
\vdots \\
 P_{s-1}= [1:(s-3)/(s-2):-1/(s-2):0].
\end{array}
\right.
\end{displaymath}
 Now let the line $L_s$ move in such a way that the $(s-1)$ points
$P_1,\ldots,P_{s-1}$ approach  the point $P= [1:0:0:0]$. In order to
find the resulting scheme of this degeneration, consider the flat
family $\{\XX_{\la}\}_{\la\in k}$, whereas  $\XX_{\la}$ is the
scheme
\begin{equation}\label{defl}
 [L_1+L_2'+\cdots+L_{s-1}'+\DD_{H,s-1}(P)]+[L_{s,\la}+2P_{1,\la}+\cdots+2P_{s-1,\la}],
 \end{equation}
 where the points $P_{1,\la},\ldots,P_{s-1,\la}$,
\begin{displaymath}
\left\{
\begin{array}{llllll}
P_{1,\la}= [1:0:-\la:0]; \\
P_{2,\la}= [1:\la:0:0]; \\
P_{3,\la}= [1:\la/2:-\la/2:0]; \\
P_{4,\la}= [1:2\la/3:-\la/3:0];\\
\vdots \\
P_{s-1,\la}= [1:\la(s-3)/(s-2):-\la/(s-2):0],
\end{array}
\right.
\end{displaymath}
 move to the point
 $P$, as well as the line  joining $L_{s,\la}$, when $\la$ tends to
 $0$. That is equivalent to setting  that
 \begin{equation}\label{ls}
I_{L_{s,\la}}= (x-y-\la t,z),
 \end{equation}
 and
 \begin{equation}\label{pi}
\left\{
\begin{array}{llllll}
I_{P_{1,\la}}= (x,x-y-\la t,z);\\
I_{P_{2,\la}}= (y,x-y-\la t,z);\\
I_{P_{3,\la}}= (x+y,x-y-\la t,z);\\
I_{P_{4,\la}}= (x+2y,x-y-\la t,z);\\
\vdots \\
I_{P_{s-1,\la}}= (x+(s-3)y,x-y-\la t,z),
\end{array}
\right.
\end{equation}

 From (\ref{defl}), observe that
\begin{equation}\label{xx}
I_{\XX_{\la}}= [I_{L_1}\cap I_{L_2'}\cap\ldots\cap I_{L_{s-1}'}\cap
(I_P^{s-1}+I_H^2)]
\end{equation}
$$ \cap [I_{L_{s,\la}}\cap I_{P_{1,\la}}^2\cap\ldots\cap
I_{P_{s-1,\la}}^2].$$

Now  to get the limiting scheme $\XX_0$, we continue as in
 \S \ref{ext}, but by using a longer and more tedious calculation,
 that we have not performed here.
 Indeed, by substituting  (\ref{li}),
(\ref{ls}), (\ref{pi}) in the right hand side of the equality
(\ref{xx}); then calculating
 in  a similar but more
complicated way to that    carried out  for  \S \ref{ext} (which can
be done by using CoCoA \cite{CoCoa});
 and finally letting $\la=0$;  we obtain that
$$I_{\XX_0}= I_{L_1}\cap I_{L_2'}\cap\ldots\cap I_{L_{s-1}'}\cap I_{L_{s,0}}\cap
(I_P^{s}+I_H^2),$$
 that is, denoting by $L_s'$ the line $L_{s,0}$,   the ideal of the  $(2,s)$-cone configuration
 $$L_1+L_2'+\cdots+L_{s-1}'+L_s'+\DD_{H,s}(P),$$
 as expected.

 Although the above proof is a very natural way to find the limiting scheme $\XX_0$, we would like
 to provide a different proof
 by finding  the limits of residual  and trace of the scheme $\XX_{\la}$  with respect to the plane
 $H$.
   In fact  to prove
   that $$\XX_0= L_1+L_2'+\cdots+L_{s-1}'+L_s'+\DD_{H,s}(P),$$
    it suffices, by construction,   to prove that
     $$Tr_H(\XX_0)= L_1+L_2'+\cdots+L_{s-1}'+L_s'$$ and
      $$Res_H(\XX_0)=\DD_{H,s-1}(P),$$
      which we will verify as follows.

 With regard to trace scheme $Tr_H(\XX_{\la})$, using (\ref{xx}), (\ref{li}), (\ref{ls}), (\ref{pi}) and noting
 that $I_P= (x,y,z), ~ I_H= (z)$, we have
\begin{eqnarray*}
I_{\XX_{\la}}+(z)&=& {\textbf{\{}}(x,z)\cap(y,z)\cap(x+y,z)\cap\ldots\cap(x+(s-3)y,z)\cap[(x,y,z)^{s-1}+(z^2)]\\
& & \ \cap (x-y-\la t,z)\cap(x,x-y-\la t,z)^2\cap(y,x-y-\la
t,z)^2\\
& & \ \cap(x+y,x-y-\la t,z)^2\cap\ldots\cap(x+(s-3)y,x-y-\la
t,z)^2{\textbf{\}}}+(z)\\
&=&  {\textbf{\{}}(x)\cap(y)\cap(x+y)\cap\ldots\cap(x+(s-3)y)\cap(x,y)^{s-1}\\
& & \ \cap (x-y-\la t)\cap(x,x-y-\la t)^2\cap(y,x-y-\la
t)^2\\
& & \ \cap(x+y,x-y-\la t)^2\cap\ldots\cap(x+(s-3)y,x-y-\la
t)^2{\textbf{\}}}+(z)\\
 &=& \textbf{(}x y(x+y)\cdots (x+(s-3)y)(x-y-\la
t)\textbf{)}+(z),
\end{eqnarray*}
and from here by setting $\la=0$, we obtain the ideal of
$Tr_H(\XX_0)$, that is
\begin{eqnarray*}
I_{\XX_{0}}+(z)&=& \textbf{(}x y(x+y)\cdots (x+(s-3)y)(x-y)\textbf{)}+(z)\\
&=& I_{L_1}\cap I_{L_2'}\cap I_{L_3'}\cap\ldots\cap I_{L_{s-1}'}\cap
I_{L_{s,0}},
\end{eqnarray*}
so we get
\begin{equation}\label{tr0}
Tr_H(\XX_0)= L_1+L_2'+L_3'+\cdots+L_{s-1}'+L_{s,0}.
\end{equation}

On the other hand  with regard to residual scheme
$Res_H(\XX_{\la})$, using (\ref{xx}), (\ref{li}), (\ref{ls}),
(\ref{pi}) and noting
 that $I_P= (x,y,z), ~ I_H= (z)$, we have
\begin{eqnarray*}
I_{\XX_{\la}}:(z)&=& \textbf{\{}(x,z)\cap(y,z)\cap(x+y,z)\cap\ldots\cap(x+(s-3)y,z)\cap[(x,y,z)^{s-1}+(z^2)]\\
& & \ \cap (x-y-\la t,z)\cap(x,x-y-\la t,z)^2\cap(y,x-y-\la
t,z)^2\\
& & \ \cap(x+y,x-y-\la t,z)^2\cap\ldots\cap(x+(s-3)y,x-y-\la
t,z)^2\textbf{\}}:(z)\\
 &=& \textbf{\{}[(x,y,z)^{s-1}+(z^2)]:(z)\textbf{\}}\\
& & \ \cap[(x,x-y-\la t,z)^2:(z)]\cap[(y,x-y-\la
t,z)^2:(z)]\\
& & \ \cap[(x+y,x-y-\la t,z)^2:(z)]\cap\ldots\cap[(x+(s-3)y,x-y-\la
t,z)^2:(z)]\\
&=& [(x,y,z)^{s-2}+(z)]
  \cap(x,x-y-\la t,z)\cap(y,x-y-\la t,z)\\
& &\cap(x+y,x-y-\la
 t,z)\cap\ldots\cap(x+(s-3)y,x-y-\la t,z)\\
 &=& [(x,y)^{s-2}+(z)]\cap\textbf{(}x y(x+y)\cdots(x+(s-3)y),x-y-\la t,z\textbf{)}\\
 &=& \textbf{\{}(x,y)^{s-2}\cap\textbf{(}x y(x+y)\cdots(x+(s-3)y),x-y-\la t\textbf{)}\textbf{\}}+(z)\\
 &=& \textbf{(}x y(x+y)\cdots(x+(s-3)y)\textbf{)}+(x,y)^{s-2}(x-y-\la t)+(z),
\end{eqnarray*}
and from here by letting $\la=0$, we obtain the ideal of
$Res_H(\XX_0)$, that is
\begin{eqnarray*}
I_{\XX_{0}}:(z)&=&
 \textbf{(}x y(x+y)\cdots(x+(s-3)y)\textbf{)}+(x,y)^{s-2}(x-y)+(z)\\
 &=& (x,y)^{s-1}+(z)\\
 &=& (x,y,z)^{s-1}+(z),
\end{eqnarray*}
so we have
\begin{equation}\label{re0}
Res_H(\XX_0)= (s-1)P|_H= \DD_{H,s-1}(P).
\end{equation}
Now putting together (\ref{tr0}) and (\ref{re0}), and denoting by
$L_s'$ the line $L_{s,0}$,  yields that
$$\XX_0= L_1+L_2'+\cdots+L_{s-1}'+L_s'+\DD_{H,s}(P),$$
which is the  $(2,s)$-cone configuration as we wanted.

By what we have seen, the proof involves two kinds of degenerations.
 The first one is the  degeneration inductive process obtained by
 successive degeneration techniques as explained above, which
  degenerates    a disjoint union $\XX$ of $s$
 lines to the scheme of type $\XX_1$.
  The second one
   is the  degeneration of  the scheme $\XX_1$ to the $(2,s)$-cone configuration $\XX_0$,
     where  the main effort of the proof goes toward in.
 These degenerations can be combined in order to give rise to the
 degeneration of $\XX$ to $\XX_0$, and that is what we wanted to
 show.
 \end{proof}
\begin{rem}
A  $(2,s)$-cone configuration as a degeneration of $s$ disjoint
lines can be obtained in a slightly different way. In fact, one can
degenerate $s$ disjoint lines  $L_1,\ldots,L_s$   to a star
configuration in a fixed plane $H$ containing $L_1$, and denote them
again by the same letters as before. Call $P$ the intersection point
of $L_1$ with $L_2$. Then let the lines $L_3,\ldots,L_s$, perhaps
one by one,  move in such a way that all the points of this star
configuration unless the point $P$, which are ${s \choose 2}-1$
points, approach  $P$.
 Now one finds that the resulting  degenerated scheme is nothing but the desired
 $(2,s)$-cone configuration $$L_1+\cdots+L_s+\DD_{H,s}(P),$$
 (see Figure \ref{fig1} for $(2,3)$-cone configuration).
\end{rem}
At the end  it is worth mentioning that although in the present
paper we only deal with the case of $(2,s)$-cone configuration as a
degeneration of $s$ disjoint lines in $\PP^3$, one might be tempted
to generalize Lemma \ref{cone} to higher dimensional ambient spaces
$\PP^n$.
 In fact, to be generally applicable in the induction argument discussed in the this section,
 the strong requirement that the
 degeneration process takes place in a flat family  is needed.
As the following lemma states, this is certainly feasible.
\begin{lem}\label{gen}
 There exists a flat family $\{\XX_{\lambda}\}_{\lambda\in k}$ of subschemes  in
$\PP^n$,  with $s,n\geq 3$, whose   general fiber is
 a disjoint  union of $s$  lines and whose   special fiber is a
 $(2,s)$-cone configuration.
\end{lem}
\begin{proof}
The proof follows an immediate generalization of the argument used
in the proof of  Lemma \ref{cone}.
\end{proof}
\section{Reduction Process}\label{pr}
We are now ready to prove the main theorem of this paper (Theorem
\ref{HH th} in Introduction):
\begin{thm}\label{mth}
 Let  $X\subset\PP^3$ be a generic union of $e$ lines.
Then $X$ has good postulation, i.e.,
$$h^0(\II_X(d)) = \max \left \{ {d+3 \choose 3}- e(d+1), 0 \right \}.$$
\end{thm}
In this section, we explain a reformulation of Theorem \ref{mth},
that is the  good postulation  statement $\HH_d$;
 next, we  describe the inductive procedure involving slightly
 different
 good postulation   statements in
lower degrees (i.e. $\HH'_{d-1}$) and lower dimensions (i.e.
$\HH''_d$)  which we will use to prove $\HH_d$.
\subsection{Auxiliary Vanishing Statements}\label{st}
Let us  illustrate our  method for proving  Theorem \ref{mth}.

In the setting of the theorem,
 as $X$ varies in a flat family,
 by the semicontinuity of cohomology \cite[III, 12.8]{Hart},
the condition of good postulation, is clearly an open condition on
the family of $X$. Hence to prove Theorem \ref{mth}, it is enough to
find any scheme  of $e$ lines, or even any scheme which is a
specialization of a flat family of $e$ lines, which has good
postulation.

 For each $d\in \NN$ consider the statement:

\vspace{0.5cm}
$(\HH_d)$:
 \emph{Assume that
 \begin{equation}\label{rq}
r= \left \lfloor {{d+3 \choose 3}}\over {d+1}\right \rfloor
 ; \ \ \ \ \ \ \
q={d+3 \choose 3} - r(d+1).
\end{equation}
 Let the scheme  $X\subset \PP^3$ be
the generic union of $r$ lines $L_1,\ldots,L_r$,  and $q$  points
$P_1,\ldots,P_q$ lying on a generic line $M$.
  Then $X$ has good
postulation, i.e.,
$$h^1(\II_X(d))= h^0(\II_X(d)) =   {d+3 \choose 3} - r(d+1)- q=
0.$$}

 Using Remark \ref{ss'}, Theorem \ref{mth} follows immediately
from the statement $\HH_d$, as explained below.  Given $d$, suppose
that $X$ is the scheme as  in the statement $\HH_d$. To prove
Theorem \ref{mth} for any $e<r$,
 if one removes the $q$ points and $r-e$
lines from $X$, one gets a scheme $Y\subset X$ consisting of $e$
disjoint lines, then $h^1(\II_X(d))=0$ implies that
$h^1(\II_{Y}(d))=0$ (by Remark \ref{ss'} (i)).  To prove Theorem
\ref{mth} for any $e>r$, if one adds the line $M$ passing through
the $q$ collinear points, as well as adds $e-r-1$ disjoint lines, to
$X$, one gets a scheme $Y\supset X$ consisting of $e$ disjoint
lines, hence $h^0(\II_X(d))=0$ gives  $h^0(\II_{Y}(d))= 0$ (by
Remark \ref{ss'} (ii)).

Looking quickly at the base cases $d=1,2$  shows that the assertions
$\HH_1$ and $\HH_2$ are quite trivial:

\begin{itemize}
 \item[$(\HH_1)$:] If $X$ is a  union of 2 skew lines in $\PP^3$, then
 $h^0(\II_X(1))=0$, which means that there is no plane containing
 $X$; (this is so obvious).
 \item[$(\HH_2)$:] If $X$ is a generic union of 3 lines and one point in $\PP^3$, then $h^0(\II_X(2))=0$,
  which means that there is no quadric surface containing $X$;
  (this is clear,  because of the
  fact that through 3 disjoint lines in $\PP^3$ there passes a unique
  quadric surface, so it is suffices to take the point outside this
  quadric).
\end{itemize}
 It remains thus just  to verify   $\HH_d$ for the cases $d\geq 3$.
 In order to show  this assertion in its  generality, we need
 to consider two further statements for each $d\geq 3$, namely
the following:

\vspace{0.5cm}
$(\HH'_{d-1})$:
 \emph{Taking  $r$  as defined in (\ref{rq}),
  assume that
\begin{equation}\label{mst}
  m= \left \lfloor {d}\over {3}\right \rfloor+1
 ; \ \ \
s={m \choose 2}+ (r-m)d- {d+2 \choose 3}; \ \ \  t= r-m-2s.
 \end{equation}
 Let $H$ be a fixed plane in $\PP^3$,  and let  $P$ be a generic point on
$H$. Let $L_1,\ldots,L_t$ be  $t$ generic lines in $\PP^3$, and let
$C_1,\ldots,C_s\subset \PP^3$ be  $s$ generic degenerate conics with
only their singular points on the plane $H$.
 If
  $$X'= (m-1)P|_H+L_1+\cdots+L_t+C_1+\cdots+C_s\subset\PP^3,$$
then
$$ h^0(\II_{X'}(d-1)) =  0.$$}

$(\HH''_{d})$:
 \emph{ Take
$q$ as defined in (\ref{rq}), and take $m,s,t$  as defined in
(\ref{mst}). Let the scheme $X''\subset \PP^2$ be the generic union
of  one cone configuration of type $m$, $s$ double points, $t$
simple points, and $q$ points lying on a generic line. Then
$$ h^0(\II_{X''}(d)) =  0.$$}

Keep in mind  the data  $r,q,m,s,t$  as  defined in (\ref{rq}) and
(\ref{mst}) in the remainder of the paper. Note that  a direct
computation yields the following relations, which are extensively
used in Sections \ref{h'} and \ref{h''},

\vspace{0.3cm}

 \hspace{-0.6cm}
 $(\ddag):$
$$
  \begin{matrix}
{\hbox{for}} \ {d\equiv 0} \ \hbox{\small{(mod 3)}}:  &
r=\frac{(d+2)(d+3)}{6};& q=0;& m= \frac{d+3}{3};& s={m-1 \choose
2};& t= {m+2 \choose 2}- 3,\\ \\
 {\hbox{for}} \ {d\equiv 1} \ \hbox{\small{(mod 3)}}: &  r=\frac{(d+2)(d+3)}{6};& q=0; &
   m= \frac{d+2}{3};& s={m \choose 2};& t= {m+1 \choose
 2},\\ \\
 {\hbox{for}} \ {d\equiv 2} \ \hbox{\small{(mod 3)}}: &  r=\frac{(d+1)(d+4)}{6};& q=\frac{d+1}{3};
 & m= \frac{d+1}{3};& s={m \choose 2};& t= {m+2 \choose 2}- 1 .\\
   \end{matrix}
$$

\vspace{0.1cm}
\subsection{$\mathbf{\HH'_{d-1} + \HH''_{d}  \Rightarrow
\HH_{d}}$  for  $\mathbf{d\geq 3}$}\label{hhh}

Our goal in this part will be to show that: if $d\geq 3$, then the
statements $\HH'_{d-1}$ and $\HH''_d$ imply $\HH_d$.  To begin, let
the scheme $X$ as in the statement $\HH_d$, where
$$X= L_1+\cdots+L_r+P_1+\cdots+P_q\subset \PP^3,$$
 is the generic union of $r$
lines $L_1,\ldots,L_r$,  and $q$  points $P_1,\ldots,P_q$ lying on
the generic line $M$.

Fix a generic plane $H\cong \PP^2$. For the purpose of getting
$h^0(\II_{X}(d)) = 0$, we wish to find a scheme $\widetilde{X}$
obtained from $X$ by  different kind of specializations and
degenerations in the most  efficient way,  in the sense  that the
desired vanishing $h^0(\II_{\widetilde{X}}(d)) = 0$ can be then
achieved. Now, setting $m,s,t$ as defined in (\ref{mst}), we
construct such $\widetilde{X}$ from $X$  by specializing the $q$
collinear points into $H$, by  degenerating $m$ lines to a
$(2,m)$-cone configuration supported in $H$, and by degenerating $s$
other pairs of lines to sundials with the property that  only their
singular points contained in $H$. To be more precise:
\begin{itemize}
\item specialize the line $M$, consequently the points
$P_1,\ldots,P_q$,  into $H$, (by abuse of notation, we will again
denote these specialized points by $P_1,\ldots,P_q$);
\item degenerate   $m$ of the lines $L_i$ to  $(2,m)$-cone
configuration $$\widehat{\CC}= \CC+ \DD_{H,m}(P),$$
 where $\CC$ is a cone configuration of type $m$ in $H$, and $P$ is
 the singular point of $\CC$, (recall that $\DD_{H,m}(P)= mP\cap 2H$);
\item degenerate the next $s$ pairs of  lines $L_i$, so that they become
   $s$ sundials $$\widehat{C}_i= C_i+ 2Q_i;\ \ \ \ \ (1\leq i\leq s),$$
   where  $C_i$ is a
   degenerate conic  and $Q_i$
   is the
singular point of $C_i$; furthermore, specialize the points $Q_i$
into $H$;
\item leave the remaining lines $L_i$, which are $t=r-m-2s$ lines,
generic outside $H$, denoted $L_1,\ldots,L_t$;
\end{itemize}
 then let
$$\widetilde{X}= \widehat{\CC}+\widehat{C}_1+\cdots+\widehat{C}_s+L_1+\cdots+L_t+P_1+\cdots+P_q\subset\PP^3.$$

 We need to show that $h^0(\II_{\X}(d))= 0,$ which consequently,
 by the semicontinuity of cohomology  \cite[III, 12.8]{Hart},
  implies that  $h^0(\II_{X}(d))= 0$.
  To do this, by Castelnuovo's inequality (Lemma \ref{cas}),
  it would be enough to show that the following vanishings
 $$h^0(\II_{Res_H(\X)}(d-1))= 0; \ \ \ \ \ h^0(H,\II_{Tr_H(\X)}(d))=0.$$

 First  we consider the residual of $\widetilde{X}$ with respect to  $H\cong \PP^2$.
 By using the fact that $Res_H(\widehat{\CC})= (m-1)P|_H$,  moreover  by  observing
 that
$Res_H(\widehat{C_i})= C_i$,  we get
 $$Res_H(\widetilde{X})= (m-1)P|_H+C_1+\cdots+C_s+L_1+\cdots+L_t\subset \PP^3.$$
 Note that   $C_1,\ldots,C_s$ are $s$ generic degenerate conics in
 $\PP^3$
  with the property that only their
 singular points lie on the plane $H$.   Then it is  quite immediate to see that
  the scheme $Res_H(\widetilde{X})$ is
  of the type $X'$ as in the statement $\HH'_{d-1}$. So this implies
  that
 proving  $h^0(\II_{Res_H(\X)}(d-1))= 0$ is precisely what we have
 to verify for  $\HH'_{d-1}$.

 Now we consider the trace of $\widetilde{X}$ on the plane $H$.
  By letting  the points $N_i=L_i\cap H$, $(1\leq i\leq t)$,  also  by using the
  facts
  that $Tr_H(\widehat{\CC})= \CC$ and that
    $Tr_H(\widehat{C}_i)$ is the double point $2Q_i|_H$ in
   $H$,  $(1\leq i\leq s)$,  we have
$$Tr_H(\widetilde{X})=
\CC+2Q_1|_H+\cdots+2Q_s|_H+N_1+\cdots+N_t+P_1+\cdots+P_q\subset
H\cong\PP^2,$$
 which is
     generic union in $H$ of the cone configuration $\CC$
 of type $m$, $s$ double points, $t$ simple points, and $q$
 collinear points.  Hence it is straightforward  to observe  that
  the scheme $Tr_H(\widetilde{X})$ is   of the type $X''$ as in
 the statement
 $\HH''_d$. It follows that
 proving  $h^0(H,\II_{Tr_H(\X)}(d))= 0$
 is exactly what we have to investigate for   $\HH''_d$.
\section{Proof of $\HH'_{d-1}$}\label{h'}
 In this section we will prove the statement $\HH'_{d-1}$ for all
 $d\geq 3$, which for convenient we state again.

\vspace{0.5cm}
 \emph{$(\HH'_{d-1})$:
 Let $d\geq 3$, and
$$r= \left \lfloor {{d+3 \choose 3}}\over {d+1}\right \rfloor; \ \ \
\ \ \  m= \left \lfloor {d}\over {3}\right \rfloor+1 ;$$
 $$s={m \choose 2}+ (r-m)d- {d+2 \choose 3}; \ \ \  \ \ \ t= r-m-2s.$$
 Let $H$ be a fixed plane in $\PP^3$, and $P$ be a generic point on $H$.
     Consider the scheme $X'$ as
     $$X'= (m-1)P|_H+L_1+\cdots+L_t+C_1+\cdots+C_s\subset\PP^3,$$
where $L_1,\ldots,L_t$  are $t$ generic lines in $\PP^3$, and
$C_1,\ldots,C_s$ are   $s$ generic degenerate conics in $\PP^3$ with
the property that  only their singular points, denoted $Q_i$ $(1\leq
i\leq s)$, lie on the plane $H$.
 Then
$$ h^0(\II_{X'}(d-1)) =  0.$$}

\vspace{0.2cm}
 Let us consider the  initial case $d= 3$.
 In this case we find that
   $$m=2; \ \ \ \ \ s=0; \ \ \ \ \ t=3,$$
   which means that
 $$X'= P+L_1+L_2+L_3\subset \PP^3,$$
 where $P$ is a generic point on the plane $H$, and $L_i$ are three
 generic lines in $\PP^3$ not lying in $H$. We specialize the line $L_1$ into $H$,
 and we again denote this specialized line by $L_1$, in addition we denote by $\Xp$
 the  scheme obtained from $X'$ by this specialization.
 We have
 $$Res_H(\Xp)= L_2+L_3\subset \PP^3,$$
 which is a union of two skew lines,
 so it is obvious that
 $$h^0(\II_{Res_H(\Xp)}(1))= 0.$$

 Also,
 $$Tr_H(\Xp)= P+L_1+N_2+N_3\subset H\cong\PP^2,$$
 where $L_i$ meets $H$ in the point $N_i$, $(i=2,3)$. Since the line
 $L_1$ is a fixed component for the curves  of
 $H^0(H,\II_{Tr_H(\Xp)}(2))$, we get
 $$h^0(H,\II_{Tr_H(\Xp)}(2))= h^0(H,\II_{Tr_H(\Xp)-L_1}(1)).$$
 Observing that  $Tr_H(\Xp)-L_1$ is made by three generic points, it trivially
 follows
 $$ h^0(H,\II_{Tr_H(\Xp)-L_1}(1))=0,$$
 and from here
 $$ h^0(H,\II_{Tr_H(\Xp)}(2))=0.$$

 Putting together $h^0(\II_{Res_H(\Xp)}(1))= 0$ and $
 h^0(H,\II_{Tr_H(\Xp)}(2))=0$, from Castelnuovo's inequality (Lemma
 \ref{cas}) we get $h^0(\II_{\Xp}(2))= 0$, therefore the semicontinuity
 of cohomology \cite[III, 12.8]{Hart} implies that
$$h^0(\II_{X'}(2))= 0.$$
Hence the case $d=3$ is done, i.e. $\HH'_2$ is proved.

\vspace{0.3cm}
 Now assume $d\geq 4$.
 The rest of the  proof will be by induction on $d$. We will investigate  separately the three cases
  ${d\equiv 0, 1, 2}$ (mod 3).
\begin{note}
 To stress the
  numerical data $m,s,t$ for  the scheme $X'$ considered in
 the statement $\HH'_{d-1}$, we will sometimes use the
 notation $X'_{(m,s,t;d)}$.
 \end{note}

\vspace{0.1cm}
\subsection{Case $\mathbf{d\equiv 0}$ (mod 3)}\label{h'0}
In this case we have, (see $(\ddag)$),
$$  m= \frac{d+3}{3}; \ \ \ \
s= {m-1 \choose 2}; \ \ \ \ t= {m+2 \choose 2}- 3.$$

 In order to the task of proving  $h^0(\II_{X'}(d-1))= 0$,
 we make a specialization of $X'$  via the plane $H$.
 Let $\Xp$ be the scheme obtained from $X'$ by
 degenerating   $m-1$ of the lines $L_i$ in such a way that they
 become a  $(2,m-1)$-cone configuration $\widehat{\CC}$,
  $$\Ch= \CC+ \DD_{H,m-1}(R),$$
 where $\CC$ is a cone configuration of type $m-1$ in $H$ and $R$
 is the singular point of $\CC$. The  remaining lines $L_i$,
 which are $t'= t-(m-1)= {m+1 \choose 2}-1$ lines, are generic not lying on $H$,
 and we  denote them by $L_1,\ldots,L_{t'}$. Then
 $$\Xp= (m-1)P|_H+ L_1+\cdots+L_{t'}+ \Ch+
 C_1+\cdots+C_s\subset \PP^3.$$

 Our goal is to
    show that $h^0(\II_{\Xp}(d-1))= 0,$
 which by the semicontinuity of cohomology  \cite[III, 12.8]{Hart} gives the desired
conclusion $h^0(\II_{X'}(d-1))= 0.$
   Now  to do so,  by Castelnuovo's inequality (Lemma \ref{cas}),
  our next task will be to show
 that
 $$h^0(\II_{Res_H(\Xp)}(d-2))= 0; \ \ \ h^0(H,\II_{Tr_H(\Xp)}(d-1))=0.$$

 First we consider  $Res_H(\Xp)$.
Since the scheme $(m-1)P|_H$ is completely contained in the plane
$H$,
 its
residual with respect to  $H$ is empty. Notice that
 $$Res_H(\Ch)= Res_H(\CC+\DD_{H,m-1}(R))= (m-2)R|_H.$$
 Letting $m'= m-1= \frac{d}{3}$,
 we  obtain
 $$Res_H(\Xp)=
 L_1+\cdots+L_{t'}+(m'-1)R|_H+C_1+\cdots+C_s\subset \PP^3,$$
 where $t'= {m'+2  \choose 2}-1$ and $s= { m' \choose 2}$. So
 it is an immediate to see that  $Res_H(\Xp)$   is
 precisely of the type   $X'_{(m',s,t';d-1)}$ as in $\HH'_{d-2}$  for the case of
 ${d-1\equiv 2}$ (mod 3), (one can check this fact simply
  by substituting  $d-1$ in  line 3 of  $(\ddag)$).
  Then  by induction
 hypothesis we have
  $$h^0(\II_{Res_H(\Xp)}(d-2)) =  0,$$
 so we are done with the residual scheme.

 Now we consider the trace scheme $Tr_H(\Xp)$.
Recalling that the singular point $Q_i$  of $C_i$ lies on the plane
$H$,  we conclude  that the degenerate conic $C_i$ meets $H$
 at $Q_i$ in degree $2$, that is a 2-dot in $H$ with
 support at $Q_i$, which we denote by $Z_i$, $(1\leq i\leq s)$.
 Moreover $L_i$ meets $H$ at the simple point $N_i$, $(1\leq i\leq t')$.
  So we get
 $$Tr_H(\Xp)= (m-1)P|_H+N_1+\cdots+N_{t'}+ \CC+
 Z_1+\cdots+Z_s\subset H\cong\PP^2.$$

 We notice that $\CC$ is a cone configuration of type $m-1$, (i.e.
 a union of $m-1$ intersecting lines passing through a single
 point), thus $\CC$ is a fixed component for the curves of
 $H^0(H,\II_{Tr_H(\Xp)}(d-1))$. Removing the fixed component $\CC$
 implies that
 \begin{equation}\label{tt}
h^0(H,\II_{Tr_H(\Xp)}(d-1))=  h^0(H,\II_{Tr_H(\Xp)-\CC}(d-m)),
\end{equation}
 which leads us to compute $h^0(H,\II_{Tr_H(\Xp)-\CC}(d-m))$, (note that by definition $d-m\geq 1$).
 By setting   $$T= (m-1)P|_H+Z_1+\cdots+Z_s,$$ we observe that $T$  is the generic union
 in $H$ of
 one $(m-1)$-multiple point and $s$ $2$-dots.
 Moreover, by  using $d= 3m-3$ and  $m\geq2$,  we see that  $d-m\geq m-2$ always holds.
 Hence   we can
 apply Lemma \ref{dot} to the scheme $T$ in degree $d-m$, and
  we then  obtain
  \begin{eqnarray*}
  h^0(H,\II_{T}(d-m))&=& \max \left \{ {d-m+2 \choose 2}- {m \choose 2}- 2s, 0 \right
  \}\\
  &=&  \max \left \{ {2m-1 \choose 2}- {m \choose 2}- 2{m-1 \choose 2}, 0 \right
  \}\\
  &=& \max \left\{ {m+1 \choose 2}-1, 0 \right\}\\
  &=&  {m+1 \choose 2}-1.
\end{eqnarray*}
 We have
 $Tr_H(\Xp)-\CC=  T+ N_1+\cdots+N_{t'},$  where
   the points $N_1,\ldots,N_{t'}$ are $t'= {m+1 \choose 2}-1$ generic points in
   $H$.
  So  the equality  $$h^0(H,\II_{T}(d-m))= {m+1 \choose 2}-1$$ implies  that
 $$h^0(H,\II_{Tr_H(\Xp)-\CC}(d-m))= 0,$$
 which is, by  (\ref{tt}), equivalent to
 $$h^0(H,\II_{Tr_H(\Xp)}(d-1))= 0.$$
This  finishes the proof.

\vspace{0.2cm}
\subsection{Case $\mathbf{d\equiv 1}$ (mod 3)}\label{h'1}
 In this case we have, (see $(\ddag)$),
$$  m= \frac{d+2}{3}; \ \ \ \
s= {m \choose 2}; \ \ \ \ t= {m+1 \choose 2}.$$
 Recall that the scheme
     $$X'= (m-1)P|_H+L_1+\cdots+L_t+C_1+\cdots+C_s\subset\PP^3,$$
is generic union of the $t$ lines $L_i$,  the $s$ degenerate conics
$C_i$ with the property that only their singular points $Q_i$ lie on
the plane $H$, and finally $(m-1)P|_H$.

\vspace{0.2cm}
 First consider the case  $d=4$.
 We have $m=2, s=1,
t=3$. (Note that we consider this case $d=4$ separately,
 because the general procedure requires that $s\geq m$ and  $m\geq 3$, while for
 $d=4$ we have $s=1$ and $m=2$). Then
 $$X'= P+L_1+L_2+L_3+C\subset\PP^3,$$
 where only  the point $P$ and the singular point $Q$ of $C$ lie on $H$.
 We need to prove $h^0(\II_{X'}(3))= 0$.
 Let $M_1,M_2$ be the two intersecting lines which form the degenerate conic $C$,
 that is $C= M_1+M_2$. Let $\Xp$ be the scheme obtained from $X'$ by
 degenerating the lines $L_1$ and $M_1$ into $H$ so that they become a
 sundial $\widehat{C'}= L_1'+M_1'+2R$ with  $R\neq Q$, where $(L_1'+M_1')$ is a
 degenerate conic in $H$ and $2R$ is  double point in $\PP^3$ with support
 at the singular point of $(L_1'+M_1')$,
  (the lines $L_2,L_3,M_2$ remain generic not lying on $H$).
 Now by the semicontinuity of cohomology it suffices to prove $h^0(\II_{\Xp}(3))= 0,$
 whereas $$\Xp= P+L_2+L_3+M_2+\widehat{C'}\subset\PP^3.$$

 By  having the fact that $Res_H(\widehat{C'})= R$ and
 $Tr_H(\widehat{C'})= L_1'+M_1'$,
 we easily get
 $$Res_H{\Xp}= L_2+L_3+M_2+R\subset\PP^3;$$
 $$Tr_H{\Xp}= P+N_2+N_3+L_1'+M_1'\subset H,$$
 where $N_2= L_2\cap H$ and $N_3=L_3\cap H$.
 The residual scheme $Res_H(\Xp)$ is  generic union of three lines
 and a simple point in $\PP^3$, so it is clear (see ($\HH_2$) in \S \ref{st})  that
\begin{equation}\label{r11}
  h^0(\II_{Res_H(\Xp)}(2)) =  0.
  \end{equation}
 Moreover, note that  every degree 3 curves lying on $H\cong \PP^2$ and passing through
 $Tr_H{\Xp}$ has to have the
  degenerate conic $(L_1'+M_1')$ as a factor. Hence  taking away this factor
  implies
  $$h^0(H,\II_{Tr_H(\Xp)}(3)) =  h^0(H,\II_{P+N_2+N_3}(1)),$$
but the right hand side, where  $P+N_2+N_3$ is  generic union of
three points in $H$,   is trivially equal to zero.
 So we get the conclusion by putting the last equality  together with
 (\ref{r11}), and then applying Castelnuovo's inequality.

\vspace{0.2cm}
  From now on we assume that  $d\geq 7$.
 This provides that   $m\geq 3$ and $s\geq m$.
 To prove the  vanishing $h^0(\II_{X'}(d-1))= 0$,
 we introduce a scheme $\Xp$ obtained from $X'$ by combining
specializations and degenerations in the following way:
\begin{itemize}
\item   Degenerate the lines $L_1,L_2$ to  sundial
 $$\widehat{C}= C+2Q,$$
 where  $C$ is a degenerate conic and $Q$ is
 its singular point, furthermore, specialize only
  the point $Q$  into $H$.
\item Consider the first $m$  degenerate conics $C_i$, and let
$M_{1,i}, M_{2,i}$ be the two lines which form  $C_i$, $(1\leq i\leq
m)$, this means
 $$C_i= M_{1,i}+ M_{2,i}; \ \ \ \ (1\leq i\leq m).$$
  Then degenerate  $C_1,\ldots, C_m$
 in such a way that the lines
 $M_{1,1},\ldots,M_{1,m}$   become a $(2,m)$-cone
 configuration $\Ch$ supported in $H$, that is
 $$\Ch= \CC+\DD_{H,m}(R),$$
  where $\CC$ is a cone configuration of
 type $m$ in $H$ and $R$ is the singular point of $\CC$.
 Leave
  the lines $M_{2,1},\ldots,M_{2,m}$ remain generic lines,
 not lying on $H$.
\end{itemize}
Now let
 $$\Xp= (m-1)P|_H+ \widehat{C}+L_3+\cdots+L_t+ \Ch+
 M_{2,1}+\cdots+M_{2,m}+C_{m+1}+\cdots+C_s\subset\PP^3.$$

 We perform the process of verifying the residual and the trace of
 the scheme $\Xp$ with respect to the plane $H$, to get $h^0(\II_{\Xp}(d-1))=
 0$, which automatically gives our required vanishing
 $h^0(\II_{X'}(d-1))=0$.

 Let us consider the residual scheme.
 Since the scheme $(m-1)P|_H$ is completely contained in the plane
 $H$, its residual with respect to $H$ is empty.
 Since the sundial $\widehat{C}= C+2Q$ has only its singular point on
 $H$, we see that
$Res_H(\widehat{C})= C$.  Moreover, by recalling that the cone
configuration $\CC$ is contained in $H$ and that $\DD_{H,m}(R)=
mR\cap 2H$, we have
 $$Res_H(\Ch)= Res_H(\CC+\DD_{H,m}(R))=(m-1)R|_H.$$
 So we obtain
 $$Res_H(\Xp)= C+L_3+\cdots+L_t+
 (m-1)R|_H+M_{2,1}+\cdots+M_{2,m}+C_{m+1}+\cdots+C_s\subset \PP^3.$$
 Observe that $Res_H(\Xp)$ is generic union of
 the lines $L_3,\ldots,L_t,M_{2,1},\ldots,M_{2,m}$, which are
 $t'= (t-2)+m= {m+1 \choose 2}-2+m= {m+2 \choose 2}-3$ lines; and the
 degenerate conics
 $C,C_{m+1},\ldots,C_s$, which are $s'= 1+(s-m)=1+{m \choose 2}-m= {m-1 \choose 2}$ degenerate conics
  (whose singular points lie on $H$); and
 finally $(m-1)R|_H$. Then it is immediately clear that   $Res_H(\Xp)$
 is of the type $X'_{(m,s',t';d-1)}$  as
 in $\HH'_{d-2}$ for the
case of
 ${d-1\equiv 0}$ (mod 3), (to check this it is enough to
  substitute   $d-1$ in  the first line of  $(\ddag)$).
  Thus  by induction
 hypothesis we get
  $$h^0(\II_{Res_H(\Xp)}(d-2)) =  0.$$

 Let us consider the trace scheme.
 On the one hand we have degenerated  the lines
 $M_{1,1},\ldots,M_{1,m}$ to the $(2,m)$-cone configuration $\Ch$
 supported in $H$, so $Tr_H(\Ch)$ will be the cone configuration $\CC$ of
 type $m$ in $H$, which we denoted by  $\CC=
 M'_{1,1}+\cdots+M'_{1,m}\subset H$.
 On the other hand the line $M_{2,i}$, $(1\leq i\leq m)$,  meets $H$ in a point which is $M'_{1,i}\cap
 M_{2,i},$ that is already  contained in $M'_{1,i}$ and then in $\CC$. So
 $$Tr_H(\Ch+M_{2,1}+\cdots+M_{2,m})= \CC.$$
We then obtain
 $$Tr_H(\Xp)= (m-1)P|_H+
 2Q|_H+N_3+\cdots+N_t+\CC+Z_{m+1}+\cdots+Z_s\subset
 H\cong\PP^2,$$
 where  $2Q|_H=Tr_H(\widehat{C})$ is a double point in $H$ with support
 at $Q$; $N_i= L_i\cap H$ is a simple point, $(3\leq i\leq t)$; and lastly
 $Z_i= C_i\cap H$
 is a 2-dot in $H$ with support at the point $Q_i$, $(m+1\leq i\leq
 s)$.

 Here our goal is  to prove   $h^0(H,\II_{Tr_H(\Xp)}(d-1))=0$.
 According to the elementary
 observation that the cone configuration $\CC$ of type $m$ is a fixed component  for the curves of
 $H^0(H,\II_{Tr_H(\Xp)}(d-1))$, this is equivalent to prove
 $$h^0(H,\II_{Tr_H(\Xp)-\CC}(d-1-m))=0.$$

 Noting  that $m= \frac{d+2}{3}$ and consequently  $d-1-m=2m-3$, we have  to show that
 $$h^0(H,\II_{Tr_H(\Xp)-\CC}(2m-3))= 0,$$
whereas
 $$Tr_H(\Xp)-\CC= (m-1)P|_H+
 2Q|_H+N_3+\cdots+N_t+Z_{m+1}+\cdots+Z_s\subset
 H\cong\PP^2.$$
  Fix a
  generic line $L\subset H$.  Consider the scheme $\T$ obtained
  from $Tr_H(\Xp)-\CC$ by specializing the points $P, Q$ and the  $m-3$
  points $N_3,\ldots,N_{m-1}$ into $L$. By abuse of notation, we will again denote these specialized
  points by the same letters as before.
  Since $\deg(\T\cap L)= (m-1)+
  2+(m-3)= 2m-2$, the line $L$ is a fixed component for the curves
  of degree $2m-3$ lying on $H$ and containing $\T$.
  Hence
 \begin{equation}\label{rl}
  h^0(H,\II_{\T}(2m-3))=
  h^0(H, \II_{Res_L(\T)}(2m-4)),
  \end{equation}
 where
 $$ Res_L(\T)= (m-2)P|_H+Q+N_m+\cdots+N_t+Z_{m+1}+\cdots+Z_s\subset H\cong \PP^2.$$
By applying Lemma \ref{dot} to the scheme $T'$,
$$T'=
(m-2)P|_H+Z_{m+1}+\cdots+Z_s\subset H,$$ which consists of one
generic $(m-2)$-multiple point and $s-m={m-1 \choose 2}-1$ generic
2-dots, in degree $2m-4$ we have
\begin{eqnarray*}
  h^0(H,\II_{T'}(2m-4))&=& \max \left \{ {2m-2 \choose 2}- {m-1 \choose 2}- 2(s-m), 0 \right
  \}\\
  &=& \max \left \{ {2m-2 \choose 2}- {m-1 \choose 2}- 2{m-1 \choose 2}+2, 0 \right
  \}\\
  &=&  \max \left \{ {m \choose 2}+2, 0 \right
  \}\\
  &=&  {m\choose 2}+2.
\end{eqnarray*}
Since   $Res_L(\T)$ is formed by the scheme  $T'$ plus the points
$Q,N_m,\ldots,N_t$, which  are $t-m+2= {m+1 \choose 2}-m+2= {m
\choose 2}+2$ generic points, it immediately follows
$$ h^0(H, \II_{Res_L(\T)}(2m-4))= 0,$$
and from here by (\ref{rl}) we  get
$$ h^0(H, \II_{\T}(2m-3))= 0,$$
 that by the semicontinuity of cohomology gives
$$ h^0(H, \II_{Tr_H(\Xp)-\CC}(2m-3))= 0,$$
so we are done.

 \vspace{0.2cm}
\subsection{Case $\mathbf{d\equiv 2}$ (mod 3)}\label{h'2}
In this case we have, (see $(\ddag)$),
$$  m= \frac{d+1}{3}; \ \ \ \
s= {m \choose 2}; \ \ \ \ t= {m+2 \choose 2}-1.$$

Note that $t\geq m$ by the definition. Now let $\Xp$ be the scheme
obtained from $X'$ by
 degenerating   $m$ of the lines $L_i$ in such a way that they
 become a  $(2,m)$-cone configuration $\widehat{\CC}$,
  $$\Ch= \CC+ \DD_{H,m}(R),$$
 where $\CC$ is a cone configuration of type $m$ in $H$ and $R$
 is the singular point of $\CC$. The  remaining lines $L_i$,
 which are $t'= t-m= {m+2 \choose 2}-1-m= {m+1 \choose 2}$ lines, are generic not lying on $H$,
 and we  denote them by $L_1,\ldots,L_{t'}$. So we get
 $$\Xp= (m-1)P|_H+ L_1+\cdots+L_{t'}+ \Ch+
 C_1+\cdots+C_s\subset \PP^3.$$

If we prove that  $h^0(\II_{\Xp}(d-1))= 0,$ then by the
semicontinuity of cohomology we have  $h^0(\II_{X'}(d-1))= 0$, and
we are done. In order to show that $h^0(\II_{\Xp}(d-1))= 0$, by
Castelnuovo's inequality we need to show that
$$h^0(\II_{Res_H(\Xp)}(d-2))= 0; \ \ \ h^0(H,\II_{Tr_H(\Xp)}(d-1))=
0.$$

 First we treat the residual of the scheme  $\Xp$
  with respect to the plane $H$.
 Since
   $Res_H((m-1)P|_H)= \emptyset$, moreover
 $Res_H(\Ch)= (m-1)R|_H,$
 we have
 $$Res_H(\Xp)=
 L_1+\cdots+L_{t'}+(m-1)R|_H+C_1+\cdots+C_s\subset \PP^3.$$
 Notice  that  $t'= {m+1 \choose 2}$ and $s= { m \choose 2}$, then
 it is  straightforward to see  that  $Res_H(\Xp)$   is
  of the type   $X'_{(m,s,t';d-1)}$ as in $\HH'_{d-2}$  for the case of
 ${d-1\equiv 1}$ (mod 3), (one can check this
  by substituting  $d-1$ in  line 2 of  $(\ddag)$).
  Thus using the induction
 hypothesis yields the desired vanishing
  $$h^0(\II_{Res_H(\Xp)}(d-2)) =  0.$$

 Now we are left with the trace scheme.
 Observe that the  degenerate conic $C_i$ meets $H$
 at the point $Q_i$, but $\deg (C_i\cap H)=2$. This observation is equivalent
 to saying that the trace of $C_i$ on $H$  is a 2-dot  with
 support at $Q_i$. Let $Z_i$ denote this 2-dot  $(1\leq i\leq s)$.
 Let  $N_i=L_i\cap H$, $(1\leq i\leq t')$.
  We get
 $$Tr_H(\Xp)= (m-1)P|_H+N_1+\cdots+N_{t'}+ \CC+
 Z_1+\cdots+Z_s\subset H\cong\PP^2.$$

  Because the cone configuration  $\CC$ of type $m$ is a fixed component for the curves of
 $H^0(H,\II_{Tr_H(\Xp)}(d-1))$, removing  $\CC$
 implies that
$$h^0(H,\II_{Tr_H(\Xp)}(d-1))=  h^0(H,\II_{Tr_H(\Xp)-\CC}(d-1-m)).$$
 Hence our goal will be to prove
 $$h^0(H,\II_{Tr_H(\Xp)-\CC}(d-1-m))=0,$$
that  can  be obtained  in  a simple way as follows.

 Consider $$T= (m-1)P|_H+Z_1+\cdots+Z_s\subset H,$$
 which is made by
 one generic $(m-1)$-multiple point and $s$ generic $2$-dots.
 First we apply
  Lemma \ref{dot} to the scheme $T$ in degree $d-1-m$,
  which gives (recall that $d=3m-1$ and $s={m \choose 2}$),
  \begin{eqnarray*}
  h^0(H,\II_{T}(d-1-m))&=& \max \left \{ {d-m+1 \choose 2}- {m \choose 2}- 2s, 0 \right
  \}\\
  &=&  \max \left \{ {2m \choose 2}- {m \choose 2}- 2{m \choose 2}, 0 \right
  \}\\
  &=& \max \left\{ {m+1 \choose 2}, 0 \right\}\\
  &=&  {m+1 \choose 2},
\end{eqnarray*}
 next by the fact that
 $Tr_H(\Xp)-\CC$ is formed by $T$ plus
  the points $N_1,\ldots,N_{t'}$, which are $t'= {m+1 \choose 2}$ generic points in
   $H$, we immediately get
 $$h^0(H,\II_{Tr_H(\Xp)-\CC}(d-1-m))= 0.$$
Now the proof is complete.
\section{Proof of $\HH''_{d}$}\label{h''}
This section is devoted to proving the statement $\HH''_d$ for all
$d\geq 3$, that is:

\vspace{0.5cm}
$(\HH''_{d})$:
 \emph{
 Let $d\geq 3$, and
$$r= \left \lfloor {{d+3 \choose 3}}\over {d+1}\right \rfloor;
 \ \ \ \ \ \ \
q={d+3 \choose 3} - r(d+1);$$
$$ m= \left \lfloor {d}\over {3}\right
\rfloor+1 ; \ \ \ \ \ s={m \choose 2}+ (r-m)d- {d+2 \choose 3}; \ \
\ \ \  t= r-m-2s.$$ Let  the scheme $X''\subset \PP^2$ be the
generic union of a cone configuration $\CC$ of type $m$;  $s$ double
points $2Q_1,\ldots,2Q_s$;  $t$ simple points $N_1,\ldots,N_t$;  and
$q$ points $P_1,\ldots,P_q$ lying on a generic line $M$, i.e.
 $$X''=
\CC+2Q_1+\cdots+2Q_s+N_1+\cdots+N_t+P_1+\cdots+P_q\subset \PP^2.$$
Then
$$ h^0(\II_{X''}(d)) =  0.$$}

\vspace{0.2cm}
 Observe  that the sections of $\II_{X''}(d)$ correspond to curves which
 have the cone configuration $\CC$ of type $m$  as  a fixed component. Taking away $\CC$, we  get
 $$ h^0(\II_{X''}(d)) =  h^0(\II_{X''-\CC}(d-m))$$
(note that by the definition we always have $d\geq m$, see
$(\ddag)$).

To prove  $h^0(\II_{X''-\CC}(d-m))=0$, we proceed  by considering
the  two cases $d\equiv 0,1$ (mod 3) simultaneously, and the case
$d\equiv 2$ (mod 3) separately.

\vspace{0.1cm}
\subsection{Case $\mathbf{d\equiv 0,1}$ (mod 3)}\label{h''0}
From $(\ddag)$ we have $q=0$, hence
$$X''-\CC= 2Q_1+\cdots+2Q_s+N_1+\cdots+N_t\subset \PP^2.$$
 Since the points $N_i$ are $t$ generic simple points, to get  $h^0(\II_{X''-\CC}(d-m))=0$
 it suffices  to show that
 \begin{equation}\label{qt}
  h^0(\II_{2Q_1+\cdots+2Q_s}(d-m))= t.
\end{equation}
 On the other hand, since the double point $2Q_i$ imposes 3 independent conditions on
 the linear system of  curves of  given degree in $\PP^2$, the expected value for
 $h^0(\II_{2Q_1+\cdots+2Q_s}(d-m))$ is
 \begin{equation}\label{ex}
 \max\left\{ {d-m+2 \choose 2}-3s, 0 \right\}.
 \end{equation}
  An easy computation by using $(\ddag)$ yields
\begin{itemize}
 \item[]  for
   $d\equiv 0$ (mod 3):  (\ref{ex}) is equal to ${m+2
   \choose 2}-3$, that is $t$;
 \item[] for $d\equiv 1$ (mod 3):  (\ref{ex}) is equal
 to ${m+1 \choose 2}$, that is $t$.
\end{itemize}
 Thus the equality  (\ref{qt}) means that the scheme
  $(2Q_1+\cdots+2Q_s)\subset\PP^2$ has good postulation in degree
  $d-m$.

From Corollary \ref{ah2}, which is an immediate  consequence of the
celebrated theorem by  Alexander and Hirschowitz  (Theorem \ref{AH
th}), we know that
 apart from  the generic union of two double points in degree 2 and
 generic union of 5 double points in degree 4, all
  generic unions of double points in $\PP^2$  have good postulation in a given degree.
 Now  in our situation, one easily checks directly (by using $(\ddag)$) that the generic union of
 $s$ double points in degree $d-m$ is not the exceptional cases of Corollary \ref{ah2}.

\vspace{0.2cm}
\subsection{Case $\mathbf{d\equiv 2}$ (mod 3)}\label{h''2}
From $(\ddag)$ we have $q=m$, which implies that
$$X''-\CC= 2Q_1+\cdots+2Q_s+N_1+\cdots+N_t+P_1+\cdots+P_m\subset \PP^2.$$
Recall that
$$  m= \frac{d+1}{3}; \ \ \ \
s= {m \choose 2}; \ \ \ \ t= {m+2 \choose 2}-1.$$

Now we have $d-m=2m-1$.  For the purpose of having
$h^0(\II_{X''-\CC}(2m-1))=0$, we make a specialization of the scheme
$X''-\CC$ via the line $M$, which already contains the points
$P_1,\ldots,P_m$.
 Let $\Xz$ be the scheme obtained from $X''-\CC$ by specializing
 the points $N_1,\ldots,N_m$ into the line $M$ (note that this is possible because $t>m$ by the construction).
 If we show that $$h^0(\II_{\Xz}(2m-1))=0,$$
 then by the semicontinuity of cohomology we are done.

 Because of the observation  that   $\deg (\Xz\cap M)= m+m= 2m$,
 the line $M$ is a fixed component for the curves of degree $2m-1$
 containing $\Xz$. Hence
 $$h^0(\II_{\Xz}(2m-1))=h^0(\II_{Res_M(\Xz)}(2m-2)),$$
 where
 $$Res_M(\Xz)= 2Q_1+\cdots+2Q_s+N_{m+1}+\cdots+N_t\subset\PP^2.$$

A quite simple consideration shows that
 the scheme $(2Q_1+\cdots+2Q_s)$ in degree $2m-2$ is never
 in the exceptional cases of Corollary \ref{ah2}.
 Therefore it has good postulation, i.e.
 \begin{eqnarray*}
 h^0(\II_{2Q_1+\cdots+2Q_s}(2m-2))&=& \max \left \{ {2m \choose 2}- 3s, 0 \right
  \}\\
  &=&  \max \left \{ {2m \choose 2}- 3{m \choose 2}, 0 \right
  \}\\
  &=& \max \left\{ {m+1 \choose 2}, 0 \right\}\\
 &=&  {m+1 \choose 2}.
\end{eqnarray*}
 By  using the fact that the points $N_{m+1},\ldots,N_t$ are
  $t-m={m+2 \choose 2}-1-m={m+1 \choose 2}$ generic points in
 $\PP^2$, the above equality gives
$$ h^0(\II_{Res_M(\Xz)}(2m-2))=  0,$$
which finishes  the proof.




\begin{thebibliography}{CGG5}

\bibitem[AB14]{AB}
T. Aladpoosh, E. Ballico, {\em Postulation of disjoint unions of
lines and a multiple point}, Rend. Sem. Mat. Univ. Politec. Torino.
 72 (3--4) (2014), 127--145.

\bibitem[Ala16]{Alad}
T. Aladpoosh, {\em Postulation of generic lines and one double line
in $\PP^n$ in view of generic lines and one multiple linear space},
arXiv:1606.02974 [math.AG] (2016).

\bibitem[AH95]{AH}
J. Alexander, A. Hirschowitz, {\em Polynomial interpolation in
several variables}, J.  Alg. Geom. 4 (2) (1995), 201--222.

\bibitem[AH00]{AH2}
J. Alexander, A. Hirschowitz, {\em An asymptotic vanishing theorem
for generic unions of multiple points}, Invent. Math. 140 (2)
(2000), 303--325.

\bibitem[Bal11]{Bal}
E. Ballico, {\em Postulation of disjoint unions of lines and a few
planes}, J. Pure Appl. Algebra, 215 (4) (2011), 597--608.

\bibitem[Bal15]{B1}
E. Ballico, {\em Postulation of disjoint unions of lines and a
multiple point II}, Mediterr. J. Mat. (2015), 1--15.

\bibitem[BD$\mathrm{S}^{+}$17]{bau}
T. Bauer, S. Di Rocco, D. Schmitz, T. Szemberg, J. Szpond, {\em On
the postulation of lines and a fat line}, arXiv:1706.02350 [math.AG]
(2017).

\bibitem[CCG10]{CCG1}
E. Carlini, M. V. Catalisano, A. V. Geramita, {\em Bipolynomial
Hilbert functions},  J. Algebra. 324 (4) (2010), 758--781.

\bibitem[CCG11]{CCG3}
E. Carlini, M. V. Catalisano, A. V. Geramita, {\em 3-dimensional
sundials}, Cent. Eur. J. Math. 9 (5) (2011), 949--971.

\bibitem[CCG16]{CCG4}
E. Carlini, M. V. Catalisano, A. V. Geramita, {\em On the Hilbert
function of lines union one non-reduced point}, Ann. Sc. Norm.
Super. Pisa Cl. Sci (5) XV (2016), 69--84.

\bibitem[CV11]{cvan}
E. Carlini, A. Van Tuyl, {\em Star configuration points and generic
plane curves}, Proc. Am. Math. Soc. 139 (12) (2011), 4181--4192.

\bibitem[CG94]{curv}
M. V. Catalisano, A. Gimigliano, {\em On curvilinear subschemes of
$\PP^2$}, J. Pure Appl. Algebra. 93 (1994), 1--14.

\bibitem[CoC04]{CoCoa}
CoCoATeam,
\newblock {{\hbox{\rm C\kern-.13em o\kern-.07em C\kern-.13em o\kern-.15em A}}}:
 a system for doing computations in Commutative Algebra,
 \newblock available at \/ {\tt http://cocoa.dima.unige.it}, 2004.

\bibitem[Ful84]{Ful}
W. Fulton, {\em Intersection theory}, Springer-Verlag, Berlin
Heidelberg New York, 1984.

\bibitem[GMR83]{GMR}
A. V. Geramita, P. Maroscia, L. G. Roberts, {\em The Hilbert
function of a reduced k-algebra}, J. Lond. Math. Soc. (2) 28 (3)
(1983), 443--452.

\bibitem[Har77]{Hart}
R. Hartshorne, {\em Algebraic geometry}, Springer-Verlag, New York,
1977.

\bibitem[HH81]{HH}
R. Hartshorne, A. Hirschowitz, {\em Droites en position g\'en\'erale
dans l'espace projectif}, in: Algebraic Geometry, La R\'abida,
(1981); in: Lecture Notes in Math, vol. 961, Springer, Berlin,
(1982), 169--188.

\bibitem[Hir81]{Hir}
A. Hirschowitz, {\em Sur la postulation g\'en\'erique des courbes
rationnelles}, Acta Math. 146 (1981), 209--230.


\end{thebibliography}
\end{document}